\documentclass[11pt,reqno,twoside]{article}

\usepackage{amssymb,amsmath,amsthm,amsbsy,mathrsfs,color}
\usepackage{enumitem,booktabs,algorithm2e,algorithmic}
\usepackage{caption,subcaption,float}
\usepackage{graphicx}
\usepackage[normalem]{ulem}
\newcommand{\stkout}[1]{\ifmmode\text{\sout{\ensuremath{#1}}}\else\sout{#1}\fi}
\usepackage[colorlinks=true]{hyperref}
\usepackage{tgtermes}
\usepackage{fullpage}

\newtheorem{theorem}{Theorem}[section]
\newtheorem{corollary}[theorem]{Corollary}
\newtheorem{lemma}[theorem]{Lemma}
\newtheorem{proposition}[theorem]{Proposition}
\newtheorem{definition}[theorem]{Definition}
\newtheorem{remark}[theorem]{Remark}

\newcommand{\R}{{\mathbb R}}
\newcommand{\N}{{\mathbb N}}
\DeclareMathOperator{\argmin}{argmin}

\newcommand{\cC}{{\mathcal C}}
\newcommand{\cD}{{\mathcal D}}
\newcommand{\cE}{{\mathcal E}}
\newcommand{\cF}{{\mathcal F}}
\newcommand{\cG}{{\mathcal G}}
\newcommand{\cH}{{\mathcal H}}
\newcommand{\cX}{{\mathcal X}}
\newcommand{\cY}{{\mathcal Y}}
\newcommand{\cK}{{\mathcal Y}}
\newcommand{\cL}{{\mathcal L}}

\newcommand{\cO}{{\mathcal O}}
\newcommand{\cP}{{\mathcal P}}

\newcommand{\cW}{{\mathcal W}}
\newcommand{\cZ}{\mathcal Z}
\newcommand{\sS}{{\mathscr S}}

\renewcommand{\top}{*}
\newcommand{\transp}{\mathrm{T}}

\newcommand{\demi}{\frac{1}{2}}
\newcommand{\ie}{{\it i.e.}\,\,}
\newcommand{\eg}{{\it e.g.}\,\,}

\newcommand{\st}{\text{~subject to~}}

\newcommand{\eqdef}{:=}

\newcommand{\dom}{\mathrm{dom}}

\newcommand{\dotp}[2]{\langle #1,\,#2 \rangle}
\newcommand{\norm}[1]{\left\|{#1}\right\|}
\newcommand{\anorm}[1]{\|{#1}\|}

\newcommand{\abs}[1]{\left |{#1}\right |}
\newcommand{\pa}[1]{\left({#1}\right)}
\newcommand{\bpa}[1]{\Big({#1}\Big)}
\newcommand{\brac}[1]{\left[{#1}\right]}
\newcommand{\bra}[1]{\left\lbrace{#1}\right\rbrace}

\newcommand{\qandq}{ \enskip \text{and} \enskip }

\newcommand{\tcb}[1]{{\color{blue}{#1}}}
\newcommand{\tcr}[1]{{\color{red}{#1}}}

\begin{document}

\title{Fast convergence of dynamical ADMM via time scaling of damped inertial dynamics}

\author{Hedy Attouch\thanks{IMAG, Univ. Montpellier, CNRS, Montpellier, France.  E-mail: 
\url{hedy.attouch@umontpellier.fr}.} 
\and Zaki Chbani\thanks{Cadi Ayyad Univ., Faculty of Sciences Semlalia, Mathematics, 40000 Marrakech, Morroco. E-mail: \url{chbaniz@uca.ac.ma}.}
\and Jalal Fadili\thanks{Normandie Universit\'e, ENSICAEN, UNICAEN, CNRS, GREYC, France.  E-mail: \url{Jalal.Fadili@greyc.ensicaen.fr}.}
\and Hassan Riahi\thanks{Cadi Ayyad Univ., Faculty of Sciences Semlalia, Mathematics, 40000 Marrakech, Morroco. E-mail: \url{h-riahi@uca.ac.ma}.}
}

\date{}

\maketitle

\begin{abstract}
In this paper, we propose in a Hilbertian setting a second-order time-continuous dynamic system with fast convergence guarantees to solve structured convex minimization problems with an affine constraint. The system is associated with the augmented Lagrangian formulation of the minimization problem. The corresponding dynamics brings into play three general time-varying parameters, each with specific properties, and which are respectively associated with viscous damping, extrapolation and temporal scaling. By appropriately adjusting these parameters, we develop a Lyapunov analysis which provides fast convergence properties of the values and of the feasibility gap. These results will naturally pave the way for developing corresponding accelerated ADMM algorithms, obtained by temporal discretization.
\end{abstract}


\bigskip



\section{Introduction}\label{sec:prel} 
Our paper is part of the active research stream that studies the relationship between continuous-time dissipative dynamical systems and optimization algorithms.
From this perspective, damped inertial dynamics offer a natural way to accelerate these systems.
An abundant literature has been devoted to the design of the damping terms, which is the basic ingredient of the optimization properties of these dynamics. 
In line with the seminal work of Polyak on the heavy ball method with friction \cite{Polyak1,Polyak2}, the first studies have focused on the case of a fixed viscous damping coefficient \cite{Alv0,AGR,AABR}.
A decisive step was taken in \cite{SBC} where the authors considered inertial dynamics with an asymptotic vanishing viscous damping coefficient. In doing so, they made the link with the accelerated gradient method of Nesterov \cite{Nest1,Nest4,BT} for unconstrained convex minimization. This has resulted in a flurry of research activity; see \eg \cite{AC1,AC2,AC2R-EECT,ACR-rescale,ACPR,ACR-subcrit,AP,APR,APR2,AD,AAD,Bot-Cest2,CD,May,SDJS,WRJ}.
 
In this paper, we consider the case of affinely \textit{constrained} convex structured minimization problems. To bring back the problem to the unconstrained case, there are two main ways: either penalize the constraint (by external penalization or an internal barrier method), or use (augmented) Lagrangian multiplier methods. 

Accounting for approximation/penalization terms within dynamical systems has been considered in a series of papers; see  \cite{ACP,BCL} and the references therein. It is a flexible approach which can be applied to non-convex problems and/or ill-posed problems, making it a valuable tool for inverse problems. Its major drawback is that in general, it requires a subtle tuning of the approximation/penalization parameter.

Here, we will consider the augmented Lagrangian approach and study the convergence properties of a second-order inertial dynamic with damping, which is attached to the augmented Lagrangian formulation of the affinely constrained convex minimization problem. The proposed dynamical system can be viewed as an inertial continuous-time counterpart of the ADMM method originally proposed in the mid-1970s and which has gained considerable interest in the recent years, in particular for solving large-scale composite optimization problems arising in data science. Among the novelties of our work, the dynamics we propose involves three parameters which vary in time. These are associated with viscous damping, extrapolation, and temporal scaling. By properly adjusting these parameters, we will provide fast convergence rates both for the values and the feasibility gap. The balance between the viscosity parameter (which tends towards zero) and the extrapolation parameter (which tends towards infinity) has already been developed in \cite{ZLC}, \cite{HHF} and \cite{ACR-Optimization-2020}, though for different problems. Temporal scaling techniques were considered in \cite{ABCR} for the case of convex minimization without affine constraint; see also \cite{ACFR,ACR-rescale,ACR-Pafa-2020}. Thus, another key contribution of this paper is to show that the temporal scaling and extrapolation can be extended to the class of ADMM-type methods with improved convergence rates. Working with general coefficients and in general Hilbert spaces allows us to encompass the results obtained in the above-mentioned papers and to broaden their scope.

\medskip 

It has been known for a long time that the optimality conditions of the (augmented) Lagrangian formulation of convex structured minimization problems with an affine constraint can be equivalently formulated as a monotone inclusion problem; see \cite{Rock1,Rock2,Rock3}. In turn, the problem can be converted into finding the zeros of a maximally monotone operator, and can therefore be attacked using inertial methods for solving monotone inclusions. In this regard, let us mention the following recent works concerning the acceleration of ADMM methods via continuous-time inertial dynamics:
\begin{enumerate}[label=$\bullet$]
\item In \cite{Bot-Cest5}, the authors proposed an inertial ADMM by making use of the inertial version of the Douglas-Rachford splitting method for monotone inclusion problems recently introduced in \cite{BCH}, in the context of concomitantly solving a convex minimization problem and its Fenchel dual; see also \cite{Goldstein,PSB,PJ,PoonLiang} in the purely discrete setting.

\item Attouch \cite{Att1} uses the maximally monotone operator which is associated with the augmented Lagrangian formulation of the problem, and specializes to this operator the inertial proximal point algorithm recently developed in \cite{AP-max} to solve general monotone inclusions. This gives rise to an inertial proximal ADMM algorithm where an appropriate adjustment of the viscosity and proximal parameters gives provably fast convergence
properties, as well as the convergence of the iterates to saddle points of the Lagrangian function. This approach is in line with \cite{AS} who considered the case without inertia. But this approach fails to achieve a fully split inertial ADMM algorithm.
\end{enumerate}

\paragraph{Contents}  
In Section~\ref{sec:dynamic_formulation}, we introduce the inertial second-order dynamical system with damping (coined \eqref{eq:trials}) which is attached to the augmented Lagrangian formulation.
In Section~\ref{sec:Lyap}, which is the main part of the paper, we develop a Lyapunov analysis to establish the asymptotic convergence properties of \eqref{eq:trials}. This gives rise to a system of inequalities-equalities which must be satisfied by the parameters of the dynamics. 
From the energy estimates thus obtained, we show in Section~\ref{Cauchy-problem} that the Cauchy problem attached to \eqref{eq:trials} is well-posed, \ie existence and possibly uniqueness of a global solution. In Section~\ref{sec:strongly_convex}, we examine the case of the uniformly convex objectives. In Section~\ref{sec:particular}, we provide specific choices of the system parameters that satisfy our assumptions and achieve fast convergence rates. This is then supplemented by preliminary numerical illustrations. Some conclusions and perspectives are finally outlined in Section~\ref{sec:conclusion}.

\section{Problem statement}\label{sec:dynamic_formulation}

Consider the structured convex optimization problem:
\begin{equation}\tag{$\cP$}\label{eq:P}
\min_{x\in\cX, \; y\in\cY} F(x,y) \eqdef f(x) + g(y) \quad \text{ subject to } Ax + By = c,
\end{equation}
where, throughout the paper, we make the following standing assumptions:
\begin{equation}\tag{$\cH_{\cP}$}\label{eq:HP}
\hspace*{-0.5cm}
\begin{cases}
\begin{tabular}{l}
$\cX,\cY,\cZ$ are real Hilbert spaces; \\
$f: \cX \rightarrow \R, \;  g: \cY \rightarrow \R$  are convex functions of class $\mathcal C^1$; \\
$A: \cX \to \cZ, B: \cY \to \cZ$ are linear continuous operators,  $c \in \cZ$;\\
The solution set of \eqref{eq:P} is non-empty.
\end{tabular}
\end{cases}
\end{equation}
Throughout, we denote by $\dotp{\cdot}{\cdot}$ and $\norm{\cdot}$ the scalar product and corresponding norm associated to any of $\cX,\cY,\cZ$, and the underlying space is to be understood from the context.

\subsection{Augmented Lagrangian formulation}
Classically, \eqref{eq:P} can be equivalently reformulated as the saddle point problem
\begin{equation}\label{eq:minmax}
\min_{(x,y)\in \cX\times \cK} \max_{\lambda\in \cZ} \cL(x, y, \lambda),
\end{equation}
where $\cL : \cX \times  \cK \times  \cZ  \rightarrow \R$ is the Lagrangian associated with \eqref{eq:minmax}
\begin{equation}\label{eq:Lag}
\cL(x, y, \lambda) \eqdef F(x,y) +\langle\lambda , Ax+By-c\rangle .
\end{equation}
Under our standing assumption~\eqref{eq:HP}, $\cL$ is convex with respect to $(x,y)\in \cX \times\cY$, and affine (and hence concave) with respect to $\lambda \in \cZ$. A pair $(x^\star,y^\star)$ is optimal for \eqref{eq:P}, and $\lambda^\star$ is a corresponding Lagrange multiplier if and only if $(x^\star, y^\star, \lambda^\star)$ is a saddle point of the Lagrangian function $\cL$, \ie for every $(x,y,\lambda) \in \cX \times \cY \times \cZ$,
\begin{equation}\label{eq:saddlepoint}
\cL(x^\star,y^\star,\lambda) \leq \cL(x^\star,y^\star,\lambda^\star) \leq \cL(x,y,\lambda^\star).
\end{equation}
We denote by $\sS$ the set of saddle points of $\cL$. The corresponding optimality conditions read
\begin{equation}\label{opt_system}
(x^\star, y^\star, \lambda^\star)\in \sS 
\Longleftrightarrow 
\begin{cases}
\nabla_x \cL(x^\star,y^\star,\lambda^\star)=0 \\
\nabla_y \cL(x^\star,y^\star,\lambda^\star)=0 \\
\nabla_\lambda \cL(x^\star,y^\star,\lambda^\star)=0
\end{cases}
\Longleftrightarrow
\begin{cases}
\nabla f(x^\star)+A^\top\lambda^\star=0 \\  
\nabla g(y^\star)+B^\top\lambda^\star=0 \\
Ax^\star+By^\star-c=0 
\end{cases},
\end{equation}
where we use the classical  notations: $\nabla f$ and $\nabla g$ are the gradients of $f$ and $g$, $A^\top$ is the adjoint operator of $A$, and similarly for $B$. The operator $\nabla_z$ is the gradient of the corresponding multivariable function with respect to variable $z$. Given $\mu > 0$, the augmented Lagrangian $\cL_\mu  : \cX \times \cK \times  \cZ  \rightarrow \R $ associated with the problem \eqref{eq:P}, is defined by
\begin{equation}\label{eq:auglag}
\cL_\mu(x, y, \lambda) \eqdef \cL(x, y, \lambda) + \frac\mu2 \|Ax + By - c\|^2.
\end{equation}
Observe that one still has 
$
(x^\star, y^\star, \lambda^\star)\in \sS 
\Longleftrightarrow 
\begin{cases}
\nabla_x \cL_\mu(x^\star,y^\star,\lambda^\star)=0, \\
\nabla_y \cL_\mu(x^\star,y^\star,\lambda^\star)=0, \\
\nabla_\lambda \cL_\mu(x^\star,y^\star,\lambda^\star)=0.
\end{cases}
$

\subsection{The inertial system \eqref{eq:trials}}

We will study the asymptotic behaviour, as $t \to +\infty$, of the inertial system:

\medskip

\boxed{
\begin{array}{rcl} \medskip
&& \text{\eqref{eq:trials}: Temporally Rescaled Inertial Augmented Lagrangian System.} \\
\hline \\
&&
\begin{cases}
\ddot x(t)+\gamma(t)\dot x(t) + b(t)\nabla_x\cL_\mu \Big(x(t),y(t),\lambda(t)+\alpha(t) \dot\lambda(t)\Big) &=0 \vspace{1mm} \\
\ddot y(t)+\gamma(t)\dot y(t) + b(t)\nabla_y\cL_\mu \Big(x(t),y(t),\lambda(t)+\alpha(t) \dot\lambda (t)\Big) &=0 \vspace{1mm}   \\
\ddot \lambda(t)+\gamma(t)\dot \lambda(t) - b(t)\nabla_\lambda \cL_\mu \Big(x(t)+\alpha(t) \dot x(t),y(t)+\alpha(t) \dot y(t),\lambda(t)\Big)&=0,\vspace{2mm} 
\end{cases}
\end{array}
}

\medskip

\noindent for $t \in [t_0,+\infty[$ with initial conditions $(x(t_0),y(t_0),\lambda(t_0))$ and $(\dot x(t_0),\dot y(t_0), \dot \lambda(t_0))$. The parameters of \eqref{eq:trials} play the following roles:
\begin{enumerate}[label=$\bullet$]
\item $\gamma(t)$ is a viscous damping parameter, 
\item $\alpha(t)$ is an extrapolation parameter,
\item $b(t)$ is attached to the temporal scaling of the dynamic.
\end{enumerate}
In the sequel, we make the following standing assumption on these parameters:
\begin{equation}\tag{$\cH_{\cD}$}\label{eq:HD}
\hspace*{-10pt}
\gamma, \alpha,  b: [t_0, +\infty[ \to \R^+ \text{ are non-negative continuously differentiable functions}.
\end{equation}

\medskip

Plugging the expression of the partial gradients of $\cL_\mu$ into the above system, the Cauchy problem associated with \eqref{eq:trials} is written as follows, where we unambiguously remove the dependence of $(x,y,\lambda)$ on $t$ to lighten the formula,
\begin{equation}\tag{\rm{TRIALS}}
\hspace*{-0.5cm}
\begin{cases}
\ddot x+\gamma (t) \dot x + b(t) \bpa{\nabla f (x) +
A^\top \brac{\lambda + \alpha(t) \dot\lambda 
+ \mu (Ax+By-c)}} &=0 \\
\ddot y+\gamma (t)\dot y + b(t)\bpa{\nabla g (y) +
B^\top \brac{\lambda + \alpha(t) \dot\lambda 
+ \mu (Ax+By-c)}}  &=0 \\
\ddot \lambda+\gamma (t)\dot \lambda - b(t)
\bpa{A(x + \alpha(t)\dot x) + B(y + \alpha(t)\dot y) -c} &= 0 \\
(x(t_0),y(t_0),\lambda(t_0)) = (x_0,y_0,\lambda_0) \qandq \\
(\dot x(t_0),\dot y(t_0),\dot \lambda(t_0)) = (u_0,v_0,\nu_0) .
\end{cases}\label{eq:trials}
\end{equation}

If in addition to \eqref{eq:HP}, the gradients of $f$ and $g$ are Lipschitz continuous on bounded sets, we will show later in Section~\ref{Cauchy-problem} that the Cauchy problem associated with \eqref{eq:trials} has a unique global solution on $[t_0,+\infty[$. Indeed, although the existence and uniqueness of a local solution follows from the standard non-autonomous Cauchy-Lipschitz theorem, the global existence necessitates the energy estimates derived from the Lyapunov analysis in the next section. The centrality of these estimates is the reason why the proof of well-posedness is deferred to Section~\ref{Cauchy-problem}. Thus, for the moment we take for granted the existence of classical solutions to \eqref{eq:trials}.

\subsection{A fast convergence result}
Our Lyapunov analysis will allow us to establish convergence results and rates under very general conditions on the parameters of \eqref{eq:trials}, see Section~\ref{sec:Lyap}. In fact, there are many situations of practical interest where such conditions are easily verified, and which will be discussed in detail in Section~\ref{sec:particular}. Thus for the sake of illustration and reader convenience, here we describe an important situation where convergence occurs with the fast rate $\cO(1/t^2)$.

\begin{theorem} \label{thm:O-1/t2}
Suppose that the coefficients of \eqref{eq:trials} satisfy
\[
\alpha(t)= \alpha_{0} t   \mbox{ with } \alpha_{0} >0, \; \gamma(t) =  \frac{\eta+\alpha_0}{\alpha_0 t}, \; b(t)= t^{\frac{1}{\alpha_0}-2} ,
\]
where $\eta > 1$. Suppose that the set of saddle points $\sS$ is non-empty and let $(x^\star,y^\star,\lambda^\star) \in \sS$. Then, for any  solution trajectory $(x(\cdot),y(\cdot),\lambda(\cdot))$ of \eqref{eq:trials}, the trajectory remains bounded, and we have the following convergence rates:
\begin{align*}
\cL(x(t),y(t),\lambda^\star) - \cL(x^\star,y^\star,\lambda^\star) &= \cO\pa{\frac{1}{t^{\frac{1}{\alpha_0}}}}, \\
\norm{Ax(t)+By(t)-c}^2 &= \cO\pa{\frac{1}{t^{\frac{1}{\alpha_0}}}} , \\
-\frac{C_1}{t^{\frac{1}{2\alpha_0}}} \leq F( x(t),y(t))-F(x^\star,y^\star) &\leq \frac{C_2}{t^{\frac{1}{\alpha_0}}} , \\
\anorm{(\dot {x}(t), \dot {y}(t), \dot{\lambda}(t)} &= \cO\pa{\dfrac1{t}} .
\end{align*}
where $C_1$ and $C_2$ are positive constants. 

In particular, for $\alpha_0 = \demi$, \ie no time scaling $b \equiv 1$, we have
\begin{align*}
\cL(x(t),y(t),\lambda^\star) - \cL(x^\star,y^\star,\lambda^\star) &= \cO\pa{\frac{1}{t^{2}}}, \\
\norm{Ax(t)+By(t)-c}^2 &= \cO\pa{\frac{1}{t^{2}}} , \\
-\frac{C_1}{t} \leq F( x(t),y(t))-F(x^\star,y^\star) &\leq \frac{C_2}{t^{2}} .
\end{align*}
\end{theorem}

For the ADMM algorithm (thus in discrete time $t=k h$, $k \in \N, h > 0$), it has been shown in \cite{DavisYin16,DavisYin17} that the convergence rate of (squared) feasibility is $\cO\pa{\frac{1}{k}}$ and that on $|F(x_k,y_k)-F(x^\star,y^\star)|$ is $\cO\pa{\frac{1}{k^{1/2}}}$. These rates were shown to be essentially tight in \cite{DavisYin16}. Our results then suggest than for $\alpha_0=\demi$, a proper discretization of \eqref{eq:trials} would lead to an accelerated ADMM algorithm with provably faster convergence rates (see \cite{Kang15,Kang13} in this direction on specific problem instances and algorithms). These discrete algorithmic issues of \eqref{eq:trials} will be investigated in a future work.

Again, for $\alpha_0 = \demi$, the $\cO\pa{\frac{1}{t^2}}$ rate obtained on the Lagrangian is reminiscent of the fast convergence obtained with the continuous-time dynamical version of the Nesterov accelerated gradient method in which the viscous damping coefficient is of the form 
$\gamma (t) = \frac{\gamma_0}{t}$ and the fast rate is obtained for $\gamma_0 \geq 3$; see \cite{ACPR,SBC}. With our notations this corresponds to $\gamma_0 = \frac{\eta+\alpha_0}{\alpha_0}$, and our choice $\alpha_0 = \demi$ entails $\gamma_0=2\eta+1 > 3$. This corresponds to the same critical value as Nesterov's but the inequality here is strict. This is not that surprising in our context since one has to handle the dual multiplier and there is an intricate interplay between $\gamma$ and the extrapolation coefficient $\alpha$.

\subsection{The role of extrapolation}
One of the key and distinctive features of \eqref{eq:trials} is that the partial gradients (with the appropriate sign) of the augmented Lagrangian function are not evaluated at $(x(t),y(t),\lambda(t))$ as it would the case in a classical continuous-time system associated to ADMM-type methods, but rather at extrapolated points. This new property will be instrumental to allow for faster convergence rates, and it can be interpreted from different standpoints: optimization, game theory, or control:

\begin{enumerate}[label=$\bullet$]
\item \textit{Optimization standpoint}: in this field, this type of extrapolation was recently studied in \cite{ACR-Optimization-2020,HHF,ZLC}. It will play a key role in the development of our Lyapunov analysis. Observe that $\alpha(t) \dot{x}(t)$ and $\alpha(t) \dot{\lambda}(t)$ point to the direction of future movement of  $x(t)$ and $\lambda(t)$. Thus, \eqref{eq:trials} involves the estimated future positions 
$x(t) + \alpha(t) \dot{x}(t)$ and $\lambda(t) + \alpha(t) \dot{\lambda}(t)$. 
Explicit discretization 
$x_k + \alpha_k (x_{k}-x_{k-1})$ and $\lambda_k + \alpha_k (\lambda_{k}-\lambda_{k-1})$ gives an extrapolation similar to the accelerated method of Nesterov. The implicit discretization reads
$x_k + \alpha_k (x_{k+1}-x_k)$ and $\lambda_k + \alpha_k (\lambda_{k+1}-\lambda_k)$. For $\alpha_k=1$, this gives $x_{k+1}$ and $\lambda_{k+1}$, which would yield implicit algorithms with associated stability properties.

\item \textit{Game theoretic standpoint}: let us think about $(x,y)$ and $\lambda$ as two players playing against each other, and shortly speaking, we identify the players with their actions. We can then see that in \eqref{eq:trials}, each player anticipates the movement of its opponent. In the coupling term, the player $(x,y)$ takes account of the anticipated position of the player
$\lambda$, which is  $\lambda(t) + \alpha(t) \dot{\lambda}(t)$, and vice versa.

\item \textit{Control theoretic standpoint}: the structure of \eqref{eq:trials} is also related to control theory and state derivative feedback. By defining $w(t)= (x(t), y(t), \lambda(t))$ the equation can be written in an equivalent way
\[
\ddot{w}(t) + \gamma (t) \dot{w}(t) = K(t,w(t), \dot{w}(t)),
\]
for an operator $K$ appropriately identified from \eqref{eq:trials} in terms of the partial gradients of $\cL_\mu$, $\alpha$ and $b$.
In this system, the feedback control term $K$, which takes the constraint into account, is not only a function of the state $w(t)$ but also of its derivative. One can consult \cite{MVHN} for a comprehensive treatment of state derivative feedback.
Indeed, we will use $\alpha(\cdot)$ as a control variable, which will turn to play an important role in our subsequent developments.
\end{enumerate}

\subsection{Associated monotone inclusion problem}
The optimality system \eqref{opt_system} can be written equivalently as
\begin{equation}\label{descrip00}
T_{\cL}(x, y, z) = 0 ,
\end{equation}
where $T_{\cL}:   \cX \times \cY \times \cZ \to \cX \times \cY \times \cZ $ is the maximally monotone operator associated with the convex-concave function $\cL$,
and which is defined by
\begin{eqnarray}
T_{\cL}(x, y, \lambda) &= &\left( \nabla_{x,y} \cL, \, -\nabla_{\lambda} \cL \right)(x, y, \lambda) \nonumber  \\
&=& \left(\nabla f(x) + A^\top \lambda, \; \nabla g(y) + B^\top \lambda, \; -( Ax +By-c)\right).\label{descrip01}
\end{eqnarray}
Indeed, it is immediate to verify that $T_{\cL}$ is monotone using \eqref{eq:HP}. Since it is continuous, it is a maximally monotone operator. Another way of seeing it is to use the standard splitting of $T_{\cL}$ as
$
T_{\cL}= T_1 + T_2
$
where 
\begin{eqnarray*}
T_1(x, y, \lambda)& =& \left(\nabla f(x) , \ \nabla g(y) , 0 \right)\\
T_2 (x, y, \lambda) &=& \left( A^\top \lambda, \  B^\top \lambda, \ -( Ax +By-c) \right).
\end{eqnarray*}
The operator $T_1 = \partial \Phi $ is nothing but gradient of the convex function $\Phi (x,y,\lambda) = f(x) + g(y)$, and therefore is maximally monotone owing to \eqref{eq:HP} (recall that convexity of a differentiable function implies maximal monotonicity of its gradient \cite{Rock1}). The operator $T_2$ is obtained by translating a linear continuous and skew-symmetric operator, and therefore it is also maximally monotone. This immediately implies that $T_{\cL}$ is maximally monotone as the sum of two maximally monotone operators, one of them being Lipschitz continuous (\cite[Lemma~2.4, page~34]{Bre1}). In turn, $\sS$ can be interpreted as the set of zeros of the maximally monotone operator $T_{\cL}$. As such, it is a closed convex subset of $\cX \times \cY \times \cZ$.

The evolution equation associated to $T_{\cL}$ is written
\begin{equation}\label{basic-dyn}
\left\{\begin{array}{lll}
\; \dot{x}(t) 
+ \nabla f  (x(t))  +   A^\top  (\lambda(t)) &=&0\vspace{2mm}\\
 \;   \dot{y}(t) 
+ \nabla g  (y(t))    
+B^\top  (\lambda(t))  &=&0 \vspace{2mm} \\
\;  \dot{\lambda}(t)   - (A(x(t)) + B(y(t)) -c)&=&0
 \end{array}\right.
\end{equation}
Following \cite{Bre1}, the Cauchy problem \eqref{basic-dyn} is well-posed, and the solution trajectories of \eqref{basic-dyn}, which define a semi-group of contractions generated by $T_{\cL}$, converge weakly in an ergodic sense to equilibria, which are the zeros of the operator $T_{\cL}$. Moreover, appropriate implicit discretization of \eqref{basic-dyn} yields the proximal ADMM algorithm.

The situation is  more complicated if we consider the corresponding inertial dynamics. Indeed, the convergence theory for the heavy ball method can be naturally extended to the case of maximally monotone cocoercive operators. Unfortunately, because of the skew-symmetric component $T_2$ in $T_{\cL}$ (when $c=0$), the operator $T_{\cL}$ is \textit{not} cocoercive. To overcome this difficulty, recent studies consider inertial dynamics where the operator $T_{\cL}$ is replaced by its Yosida approximation, with an appropriate adjustment of the Yosida parameter; see \cite{AP-max} and \cite{Att1} in the case of the Nesterov accelerated method. However, such an approach does not achieve full splitting algorithms, hence requiring an additional internal loop.

\section{Lyapunov analysis}\label{sec:Lyap}
Let $(x^\star,y^\star) \in \cX \times \cK$ be a solution  of \eqref{eq:P}, and denote by $F^\star \eqdef F(x^\star,y^\star)$ the optimal value of \eqref{eq:P}. For the moment, the variable $\lambda^\star$  is chosen arbitrarily in  $\cZ $. We will then be led to specialize it. Let $t \mapsto (x(t),y(t),\lambda(t))$ be a solution trajectory of  \eqref{eq:trials} defined for $t\geq t_0$. It is supposed to be a classical solution, \ie of class $\cC^2$. We are now in position to introduce the function $t \in  [t_0, +\infty[\; \mapsto \cE(t)\in \R$ that will serve as a Lyapunov function,
\begin{eqnarray}
&&\cE(t)\eqdef \delta^2(t)b(t)\Big( \cL_\mu(x(t), y(t), \lambda^\star)- \cL_\mu(x^\star, y^\star, \lambda^\star)\Big)+ \frac{1}{2}\norm{v(t)}^{2}  \label{eq:lyapcont}\\ 
 && \qquad\qquad + \frac12\xi(t)\|(x(t), y(t), \lambda(t))-(x^\star, y^\star, \lambda^\star)\|^2, \nonumber
\vspace{4mm}\\
&& v(t)\eqdef \sigma (t)\Big((x(t), y(t), \lambda(t))-(x^\star, y^\star, \lambda^\star)\Big)+\delta(t)(\dot x(t), \dot y(t), \dot \lambda(t)). \label{eq:lyapcont_b} 
\end{eqnarray}

The coefficient $\sigma(t)$ is non-negative and will be adjusted later, while $\delta(t), \xi(t)$ are explicitly defined by the following formulas:
\begin{equation}\label{basic_choice_0}
\begin{cases}
\delta(t) \eqdef \sigma(t) \alpha(t), \\
\xi(t) \eqdef \sigma(t)^2\Big(\gamma(t)\alpha (t)-\dot \alpha (t)-1 \Big) -2 \alpha(t) \sigma(t) \dot \sigma(t)
\end{cases}
\end{equation}
This choice will become clear from our Lyapunov analysis.
To guarantee that $\cE$ is a Lyapunov function for the dynamical system \eqref{eq:trials}, the following conditions on the coefficients $\gamma, \,  \alpha,  \, b, \, \sigma$ will naturally arise from our analysis:
\begin{center}
\begin{tabular}{|c|}\hline
Lyapunov system of inequalities/equalities on the parameters. \\\hline
\parbox{\textwidth}{
\begin{enumerate}[label=($\cG_{\arabic*}$),itemindent=10ex]
\item $\sigma(t)\bpa{\gamma(t)\alpha (t)-\dot \alpha (t) -1} -2 \alpha(t) \dot \sigma(t) \geq 0$, \label{cond:G1}
\item $\sigma (t)\bpa{\gamma(t)\alpha (t)- \dot \alpha (t) - 1} - \alpha(t)\dot{\sigma} (t) \geq 0$,  \label{cond:G2}
\item $-\frac{d}{dt}\brac{\sigma(t) \pa{\sigma (t)\bpa{\gamma (t)\alpha(t) - \dot\alpha (t)} -2 \alpha (t)\dot\sigma (t)}} \geq 0$, \label{cond:G3}
\item $\alpha (t)\sigma (t)^2 b(t) -\frac{d}{dt}\left(\alpha^2 \sigma^2 b\right)(t)= 0$. \label{cond:G4}
\end{enumerate}}\\\hline
\end{tabular}
\end{center}
Observe that condition \ref{cond:G1} automatically ensures that $\xi(t)$ is a non-negative function.
In most practical situations (see Section~\ref{sec:particular}), we will take $\sigma$ as a non-negative constant, in which case \ref{cond:G1} and \ref{cond:G2} coincide, and thus conditions \ref{cond:G1}--\ref{cond:G4} reduce to a system of three differential inequalities/equalities involving only the coefficients $(\gamma,\alpha,b)$ of the dynamical system \eqref{eq:trials}.


%
\if
{
\begin{equation} \label{basic_choice_0}
\delta (t)= \sigma(t) \alpha (t), 
\end{equation}
\begin{equation}\label{def:xi_0}
\xi (t)\eqdef\sigma (t)^2\Big(\dot \alpha (t)+\gamma(t)\alpha (t)-1 \Big) -2 \alpha(t) \sigma(t) \dot \sigma(t).
\end{equation}
}
\fi
\if
{
In the part $(a)$ of the theorem, we only assume that the solution set of \eqref{eq:P} is non-empty, and  only get an upper bound of $F(x(t),y(t))-F^\star $.  
In the part $(b)$ of the theorem we assume that the solution set
$\sS$ of equilibria of the saddle value problem is 
non-empty,  then get  an upper bound of the absolute value $| F(x(t),y(t))-F^\star| $.
}
\fi

\subsection{Convergence rate of the values}

By relying on a Lyapunov analysis with the function $\cE$, we are now ready to state our first main result.

\begin{theorem}\label{ACFR,rescale}
Assume that \eqref{eq:HP} and \eqref{eq:HD} hold. Suppose that the growth conditions  
\ref{cond:G1}--\ref{cond:G4} on the parameters $(\gamma, \alpha, \sigma, b)$ of \eqref{eq:trials} are satisfied for all $t\geq t_0$. 
Let $t\in [t_0, +\infty[  \mapsto (x(t),y(t),\lambda(t))$ be a solution trajectory of  \eqref{eq:trials}. Let $\cE$ be the function defined in \eqref{eq:lyapcont}-\eqref{eq:lyapcont_b}. Then the following holds:
\begin{enumerate}[label=(\arabic*)] 
\item $\cE$ is a non-increasing function, and for all $t\geq t_0$ \label{theoACFRrescale:itema}
\[
F(x(t),y(t))-F^\star  = \cO\pa{\frac{1}{\alpha(t)^2\sigma(t)^2 b(t)}} .
\]

\item Suppose moreover that $\sS$, the set of saddle points of $\cL$ in \eqref{eq:minmax} is non-empty, and let $(x^\star,y^\star,\lambda^\star) \in \sS$. Then for all $t\geq t_0$, the following rates and integrability properties are satisfied: \label{theoACFRrescale:itemb}
\begin{enumerate}[label=(\roman*),itemindent=10ex]
\item \label{theoACFRrescale:itembi}
$
0 \leq \cL(x(t),y(t),\lambda^\star) - \cL(x^\star,y^\star,\lambda^\star)=\cO\pa{\frac {1}{\alpha(t)^2\sigma(t)^2 b(t)}};
$
\item \label{theoACFRrescale:itembii}
$
\anorm{Ax(t)+By(t)-c}^2=\cO\pa{\frac {1}{\alpha(t)^2\sigma(t)^2 b(t)}};
$
\item there exists positive constants $C_1$ and $C_2$ such that \label{theoACFRrescale:itembiii}
\[
-\frac{C_1}{\alpha(t)\sigma(t) \sqrt{b(t)}} \leq F(x(t),y(t))-F^\star \leq \frac{C_2}{\alpha(t)^2\sigma(t)^2 b(t)};
\]
\item \label{theoACFRrescale:itembiv}
$
\displaystyle{\int_{t_0}^{+\infty} \alpha(t)\sigma(t)^2b(t)\anorm{Ax(t)+By(t)-c}^2 dt <+\infty} ;
$
\item \label{theoACFRrescale:itembv}
$
\displaystyle{\int_{t_0}^{+\infty}
k(t) \anorm{(\dot {x}(t), \dot {y}(t), \dot{\lambda}(t))}^2 dt <+\infty},
$
where 
\[
k(t)=  \alpha (t)\sigma(t) \bpa{\sigma(t)\pa{\gamma(t)\alpha (t) - \dot \alpha(t) - 1} - \alpha(t) \dot \sigma (t)}.
\]
\end{enumerate}
\end{enumerate}
\end{theorem}

\begin{proof}
To lighten notation, we drop the dependence on the time variable $t$. Recall that $(x^\star,y^\star)$ is a solution of \eqref{eq:P} and $\lambda^\star$ is an arbitrary vector in $\cZ$. Let us define
\[
w \eqdef (x, y, \lambda), \quad w^\star\eqdef (x^\star, y^\star, \lambda^\star),\quad \cF_\mu(w) \eqdef \cL_\mu(x, y, \lambda^\star)- \cL_\mu(x^\star, y^\star, \lambda^\star).
\]
With these notations we have (recall \eqref{eq:lyapcont} and \eqref{eq:lyapcont_b})
\begin{eqnarray*}  
&& v=\sigma( w-w^\star)+\delta\dot w,\, \\
&& \nabla \cF_\mu(w)=(\nabla_x  \cL_\mu(x, y, \lambda^\star), \nabla_y  \cL_\mu(x, y, \lambda^\star),0)\\
&& \cE= \delta^2b\cF_\mu(w) + \frac{1}{2}\norm{v}^{2} +\frac12\xi\anorm{w-w^\star}^2.
\end{eqnarray*}
Differentiating $\cE$ gives
\begin{equation}\label{der-E}
\dfrac{d}{dt}\cE=\dfrac{d}{dt}(\delta^2b) \cF_\mu(w)+ \delta^2b \dotp{\nabla \cF_\mu(w)}{\dot{w}}+ \dotp{v}{\dot{v}} + \frac12\dot\xi\|w-w^\star\|^2 + \xi \dotp{w-w^\star}{\dot{w}}.
\end{equation}
Using the constitutive equation in \eqref{eq:trials}, we have
\begin{align*}
\dot{v} 
& = \dot \sigma (w-w^\star) + (\sigma  + \dot\delta )\dot w + \delta\ddot w \\
& = \dot \sigma (w-w^\star) + (\sigma  + \dot\delta )\dot w - \delta \pa{\gamma\dot w + b K_{\mu,\alpha}(w)} \\
& = \dot \sigma (w-w^\star) + (\sigma  + \dot\delta - \delta \gamma)\dot w - \delta b K_{\mu,\alpha}(w) ,
\end{align*}
where  the operator $K_{\mu,\alpha}:  \cX \times \cY \times \cZ \to \cX \times \cY \times \cZ $  is defined by 
\begin{equation*}
K_{\mu,\alpha}(w) \eqdef
\begin{bmatrix}
\nabla_x\cL_\mu(x,y,\lambda+\alpha \dot\lambda )\\  
\nabla_y\cL_\mu(x,y,\lambda+\alpha \dot\lambda )\\
-\nabla_\lambda \cL_\mu(x+\alpha \dot x,y+\alpha \dot y,\lambda) 
\end{bmatrix}
\end{equation*}
Elementary computation gives
\begin{equation*}
K_{\mu,\alpha}(w) = \nabla \cF_\mu(w) +
\begin{bmatrix}
A^\top(\lambda-\lambda^\star+\alpha \dot\lambda ) \\  
B^\top(\lambda-\lambda^\star+\alpha \dot\lambda ) \\  
-A(x+\alpha \dot x)-B(y+\alpha \dot y)+c
\end{bmatrix}
\end{equation*}
According to the above formulas for $v$, $\dot v$ and $K_{\mu,\alpha}$, we get
\begin{eqnarray*}
\dotp{v}{\dot{v}} 
&=& \dotp{\dot \sigma (w-w^\star) + (\sigma  + \dot \delta  - \delta \gamma)\dot w - \delta b K_{\mu,\alpha}(w)}{\sigma( w-w^\star)+\delta\dot w}\\
&=& \sigma \dot\sigma\anorm{w-w^\star}^2 + \pa{\delta\dot\sigma+\sigma(\sigma+\dot\delta-\delta\gamma)}\dotp{\dot w}{w-w^\star}  + \delta\pa{\sigma+\dot\delta-\delta\gamma} \anorm{\dot w}^2\\
&& -\delta b\brac{\sigma\dotp{\nabla\cF_\mu(w )}{w-w^\star} + \delta \dotp{\nabla\cF_\mu(w)}{\dot w}}\\
&& -\delta b\brac{\sigma \dotp{\lambda-\lambda^\star+\alpha \dot\lambda}{Ax-Ax^\star} + \delta \dotp{\lambda-\lambda^\star+\alpha \dot\lambda}{A\dot x}}\\
&& -\delta b\brac{\sigma \dotp{\lambda-\lambda^\star+\alpha \dot\lambda}{By-By^\star} + \delta \dotp{\lambda-\lambda^\star+\alpha \dot\lambda}{B\dot y}}\\
&& +\delta b \sigma \dotp{A(x+\alpha \dot x)+B(y+\alpha \dot y)-c}{\lambda-\lambda^\star} \\
&&+ \delta^2 b \dotp{A(x+\alpha \dot x)+B(y+\alpha \dot y)-c}{\dot \lambda}.
\end{eqnarray*}
Let us insert this expression in \eqref{der-E}. We first observe that the term $\dotp{\nabla\cF_\mu(w)}{\dot w}$ appears twice but with opposite signs, and therefore cancels out. Moreover, the coefficient of $\langle \dot w,w-w^\star\rangle$  becomes
$\xi+ \delta\dot\sigma -\sigma(\gamma\delta -\dot \delta- \sigma)$. Thanks to the choice of $\delta$ and $\xi$ devised in \eqref{basic_choice_0}, the term $\dotp{\dot w}{w-w^\star}$ also disappears. We recall that by virtue of \ref{cond:G1}, $\xi$ is non-negative, and thus so is the last term in $\cE$. Overall, the formula \eqref{der-E} simplifies to
\begin{equation}\label{basic-Lyap1}
\begin{split}
\dfrac{d}{dt}\cE 
&= \dfrac{d}{dt}(\delta^2b) \cF_\mu(w)+ \pa{\frac12\dot\xi + \sigma\dot\sigma} \anorm{w-w^\star}^2  + \delta\pa{\sigma+\dot\delta-\delta\gamma}\anorm{\dot w}^2 \\
&-\delta b \sigma \dotp{\nabla\cF_\mu(w)}{w-w^\star} - \delta  b \cW , 
\end{split}
\end{equation}
where 
\begin{eqnarray*}
\cW
&\eqdef& \sigma \dotp{\lambda-\lambda^\star+\alpha \dot\lambda}{Ax-Ax^\star} + \delta \dotp{\lambda-\lambda^\star+\alpha \dot\lambda}{A\dot x}\\
&& + \sigma \dotp{\lambda-\lambda^\star+\alpha \dot\lambda}{By-By^\star} + \delta \dotp{\lambda-\lambda^\star+\alpha \dot\lambda}{B\dot y}\\
&& - \sigma \dotp{A(x+\alpha \dot x)+B(y+\alpha \dot y)-c}{\lambda-\lambda^\star} \\
&& - \delta \dotp{A(x+\alpha \dot x)+B(y+\alpha \dot y)-c}{\dot \lambda}.
\end{eqnarray*}
Since  $(x^\star,y^\star) \in \cX\times\cK $ is a solution of \eqref{eq:P}, we obviously have $A x^\star + B y^\star =c$. Thus, $\cW$ reduces to
\begin{eqnarray*}
\cW
&=& \sigma \dotp{Ax+By-c}{\lambda-\lambda^\star+\alpha \dot\lambda} + \delta \dotp{A\dot x + B\dot y}{\lambda-\lambda^\star+\alpha \dot\lambda}\\
&& - \sigma \dotp{Ax+By-c}{\lambda-\lambda^\star} - \sigma\alpha\dotp{A\dot x + B\dot y}{\lambda-\lambda^\star}\\
&& - \delta \dotp{Ax+By-c}{\dot\lambda} - \delta\alpha\dotp{A\dot x + B\dot y}{\dot\lambda} \\
&=& (\sigma \alpha - \delta)\bpa{\dotp{Ax+By-c}{\dot \lambda} - \dotp{A \dot x + B \dot y}{\lambda-\lambda^\star}}.
\end{eqnarray*}
Since it is difficult to control the sign of the above expression, the choice of $\delta$ in \eqref{basic_choice_0} appears natural, which entails $\cW =0$. 

On the other hand, by convexity of $\cL(\cdot,\cdot,\lambda^\star)$, strong convexity of  $\frac{\mu}{2}\anorm{\cdot - c}^2$, the fact that $Ax^\star+By^\star = c$ and $\cF_\mu(w^\star)=0$, it is straightforward to see that
\[ 
-\cF_\mu(w) - \frac{\mu}{2}\norm{Ax(t)+By(t) - c}^2 \geq \dotp{\nabla \cF_\mu(w)}{w^\star-w}.
\]
Collecting the above results, \eqref{basic-Lyap1} becomes
\begin{eqnarray}\label{basic-lyap-1}
&&\dfrac{d}{dt}\cE 
+ \pa{\delta  b \sigma-\dfrac{d}{dt}(\delta^2b)}\cF_\mu(w)  \\
&&\leq  \pa{\frac12\dot\xi +  \sigma\dot\sigma} \anorm{w-w^\star}^2  + \delta\pa{\sigma+\dot\delta-\delta\gamma}\anorm{\dot w}^2 - \frac{\delta  b \sigma \mu}{2}\norm{Ax(t)+By(t) - c}^2. \nonumber
\end{eqnarray}
Since $\delta$ is non-negative ($\sigma$ and $\alpha$ are), and in view of \ref{cond:G2}, the coefficient of the second term in the right hand side \eqref{basic-lyap-1} is non-positive. The same conclusion holds for the coefficient of the first term since its non-positivity is equivalent to \ref{cond:G3}. 
%
Therefore, inequality \eqref{basic-lyap-1} implies
\begin{equation}\label{basic-Liap-22}
\dfrac{d}{dt}\cE +\left( \delta  b \sigma-\dfrac{d}{dt}(\delta^2b) \right)\cF_\mu(w) \leq 0 .
\end{equation}
%
The sign of $\cF_\mu(w)$ is unknown for arbitrary $\lambda^\star$. This is precisely where we invoke \ref{cond:G4} which is equivalent to
\begin{equation*}\label{def:sigma3}
\delta  b \sigma-\dfrac{d}{dt}(\delta^2b)= 0.
\end{equation*}		
\begin{enumerate}[label=(\arabic*)]
\item Altogether, we have shown so far that \eqref{basic-Liap-22} eventually reads, for any $t\geq t_0$,
\begin{equation}\label{Lypa_0}
\dfrac{d}{dt}\cE(t) \leq 0 ,
\end{equation}
\ie $\cE$ is non-increasing as claimed. Let us now turn to the rates.

$\cE$ being non-increasing entails that for all $t\geq t_0$
\begin{equation}\label{Lypa_decr}
\cE(t) \leq \cE(t_0) .
\end{equation}
Dropping the non-negative terms $\frac{1}{2}\norm{v(t)}^{2}$ and $\frac12\xi(t)\|w(t)-w^\star\|^2$  entering $\cE$, and according to the definition of $\cL_\mu$, we obtain that, for all $t\geq t_0$
\begin{eqnarray}
&&\delta(t)^2b(t)\bpa{\cL_\mu(x(t), y(t), \lambda^\star)- \cL_\mu(x^\star, y^\star, \lambda^\star)} \label{rat-conv-L1}\\
&&= \delta(t)^2b(t) \bpa{\cL(x(t), y(t), \lambda^\star)- \cL(x^\star, y^\star, \lambda^\star) + \frac{\mu}{2}\norm{Ax(t)+By(t) - c}^2} \leq \cE(t_0) . \nonumber
\end{eqnarray}
Dropping again the quadratic term in \eqref{rat-conv-L1}, we obtain
\begin{align*}
\delta(t)^2b(t) &\bpa{F(x(t),y(t))-F^\star + \dotp{\lambda^\star}{Ax(t)+By(t)-c}} \\ 
&\leq  \delta^2(t_0)b(t_0)\Big(F(x(t_0),y(t_0))-F^\star + \dotp{\lambda^\star}{Ax(t_0)+By(t_0)-c} \\
&+\frac{\mu}{2}\anorm{Ax(t_0)+By(t_0)-c}^2\Big)+ \frac{1}{2}\anorm{v(t_0)}^{2}\\
&+\frac12\xi(t_0)\anorm{(x(t_0), y(t_0), \lambda(t_0))-(x^\star, y^\star, \lambda^\star)}^2\\
&\leq \delta^2(t_0)b(t_0)\anorm{\lambda^\star}\anorm{Ax(t_0)+By(t_0)-c} + C_0,
\end{align*}
where $C_0$ is the non-negative constant
\begin{multline}\label{eq:C0}
C_0 
= \delta^2(t_0)b(t_0)\bpa{\abs{F(x(t_0),y(t_0))-F^\star}
+\frac{\mu}{2}\anorm{Ax(t_0)+By(t_0)-c}^2} \\
+ \frac{1}{2}\anorm{v(t_0)}^{2} + \frac12\xi(t_0)\anorm{(x(t_0), y(t_0), \lambda(t_0))-(x^\star, y^\star, \lambda^\star)}^2 .
\end{multline}
When $Ax(t)+By(t)-c = 0$, we are done by taking, \eg $\lambda^\star = 0$ and $C > C_0$.  Assume now that $Ax(t)+By(t)-c \neq 0$. Since $ \lambda^\star$ can be freely chosen in $\cZ$, we take it as the unit-norm vector
\begin{equation}\label{eq:lambdaunit}
\lambda^\star =  \frac{Ax(t)+By(t)-c}{\anorm{Ax(t)+By(t)-c}} .
\end{equation}
We therefore obtain
\begin{equation}\label{eq:estimatelyap}
\delta(t)^2b(t) \bpa{F(x(t),y(t))-F^\star + \anorm{Ax(t)+By(t)-c}} \leq C ,
\end{equation}
where $C > \delta^2(t_0)b(t_0)\norm{Ax(t_0)+By(t_0)-c} + C_0$. Since the second term in the left hand side is non-negative, the claimed rate in \ref{theoACFRrescale:itema} follows immediately.
%

\item Embarking from \eqref{rat-conv-L1} and using \eqref{eq:saddlepoint} since $(x^\star,y^\star,\lambda^\star) \in \sS$, we have the rates stated in \ref{theoACFRrescale:itembi} and \ref{theoACFRrescale:itembii}.

%
To show the lower bound in \ref{theoACFRrescale:itembiii}, observe that the upper-bound of \eqref{eq:saddlepoint} entails that
\begin{equation}\label{basic_minoration}
F(x(t), y(t)) \geq F(x^\star, y^\star)  - \dotp{Ax(t)+ By(t) - c}{\lambda^\star} .
\end{equation}
Applying Cauchy-Schwarz inequality, we infer
\[
F(x(t), y(t)) \geq F(x^\star, y^\star) - \anorm{\lambda^\star}\anorm{Ax(t)+ By(t)-c}.
\]
We now use the estimate \ref{theoACFRrescale:itembii} to conclude. Finally the integral estimates of the feasibility \ref{theoACFRrescale:itembiv} and velocity \ref{theoACFRrescale:itembv} are obtained by integrating \eqref{basic-lyap-1}. \qed  
\end{enumerate} 
\end{proof}

\if
{

\begin{remark}
{~}\medskip
\begin{enumerate}
\item From \eqref{basic-lyap-1} and \ref{cond:G3}, one has also that $\pa{\frac12\dot\xi +  \sigma\dot\sigma} \anorm{w-w^\star}^2 \in L^1([t_0,+\infty[)$. However, as we will argue in Section~\ref{sec:particular}, our conditions on the parameters $(\gamma,\alpha,b,\sigma)$ are easily verifiable when $\sigma$ is taken as a positive constant, and $\gamma\alpha - \dot{\alpha}$ a constant larger than. In this case, \ref{cond:G3} is actually an equality and the above integral estimate becomes vacuous.
\item {\tcr{Jalal:~ If one can prove that the limit of $\norm{w(t)-w^\star}$ exists, then by appropriately strengthening \ref{cond:G2} and \ref{cond:G4}, the rate in Theorem~\ref{ACFR,rescale}\ref{theoACFRrescale:itembii} and the lower rate in \ref{theoACFRrescale:itembiii} can be improved to $o(.)$.}}
\item {\tcr{Jalal:~Convergence of the iterates is completely open at this point and is a challenging problem for the system \eqref{eq:trials}}.}
\end{enumerate}
\end{remark}

}
\fi

\subsection{Boundedness of the trajectory and rate of the velocity}

We will further exploit the Lyapunov analysis developed in the previous section to assert additional properties on the iterates and velocities.

\begin{theorem}\label{ACFR_rescale_boundedness}
Suppose the assumptions of Theorem~\ref{ACFR,rescale} hold.
Assume also that $\sS$, the set of saddle points of $\cL$ in \eqref{eq:minmax} is non-empty, and let $(x^\star, y^\star, \lambda^\star) \in \sS$. 
Then, each solution trajectory  $t\in [t_0, +\infty[  \mapsto (x(t),y(t),\lambda(t))$ of \eqref{eq:trials} satisfies the following properties: 
\begin{enumerate}[label=(\arabic*)]
\item There exists a positive constant $C$ such that, for all $t \geq t_0$
\begin{eqnarray*}
&&\anorm{(x(t), y(t), \lambda(t))-(x^\star, y^\star, \lambda^\star)}^2 \leq \frac{C}{\sigma (t)^2\Big(\gamma(t)\alpha (t)-\dot \alpha (t)-1 \Big) -2 \alpha(t) \sigma(t) \dot \sigma(t)}  \\
&&\norm{(\dot {x}(t), \dot {y}(t), \dot{\lambda}(t))} \leq \frac{C}{\alpha(t)\sigma(t)}\pa{1+\sqrt{\frac{\sigma(t)}{\sigma (t)\pa{\gamma(t)\alpha (t)-\dot \alpha (t)-1} -2 \alpha(t)  \dot \sigma(t)}}}.
\end{eqnarray*}
\item 
\label{ACFR_rescale_boundedness:itemii}
If $\sup_{t \geq t_0} \sigma(t) < +\infty$ and \ref{cond:G1} is strengthened to  
\begin{enumerate}[label=($\cG_{\arabic*}^{+}$),itemindent=10ex]
\item $\inf_{t\geq t_0} \sigma(t)\bpa{\sigma (t)\pa{\gamma(t)\alpha (t)-\dot \alpha (t)-1} - 2 \alpha(t) \dot \sigma(t)} >0$, \label{cond:G1+}
\end{enumerate}
then
\begin{align*}
\sup_{t\geq t_0} \norm{(x(t),  y(t), \lambda(t))}  < +\infty \qandq \norm{(\dot {x}(t), \dot {y}(t), \dot{\lambda}(t))} = \cO\pa{\frac{1}{\alpha(t)\sigma(t)}} .
\end{align*}
If moreover, 
\begin{enumerate}[label=($\cG_{\arabic*}$),itemindent=10ex,start=5]
\item $\inf_{t \geq t_0} \alpha(t) > 0$, \label{cond:G5}
\end{enumerate}
then
\[
\sup_{t\geq t_0} \norm{(\dot {x}(t), \dot {y}(t), \dot{\lambda}(t))} < +\infty .
\]
\end{enumerate}
\end{theorem}

\begin{proof}
We start from \eqref{Lypa_decr} in the proof of Theorem~\ref{ACFR,rescale}, which can be equivalently written
\begin{equation*}\label{eq:lyapcont_1}
\begin{split}
&\delta^2(t)b(t)\pa{\cL_\mu(x(t), y(t), \lambda^\star)- \cL_\mu(x^\star, y^\star, \lambda^\star)} + \frac{1}{2}\anorm{v(t)}^{2} \\ 
&+\frac12\xi(t)\anorm{(x(t), y(t), \lambda(t))-(x^\star, y^\star, \lambda^\star)}^2 \leq \cE(t_0).
\end{split} 
\end{equation*}
Since $(x^\star,y^\star,\lambda^\star) \in \sS$, the first term is non-negative by \eqref{eq:saddlepoint}, and thus
\begin{equation*}
\frac{1}{2}\norm{v(t)}^{2} + \frac12\xi(t)\anorm{(x(t), y(t), \lambda(t))-(x^\star, y^\star, \lambda^\star)}^2 \leq \cE(t_0).
\end{equation*}
Choosing a positive constant $C \geq \sqrt{2\cE(t_0)}$, we immediately deduce that for all $t\geq t_0$  
\begin{equation}\label{basic_maj_3}
\anorm{(x(t), y(t), \lambda(t))-(x^\star, y^\star, \lambda^\star)} \leq \frac{C}{\sqrt{\xi(t)}} \quad \text{and} \quad  \anorm{v(t)} \leq C .
\end{equation}
Set $z(t)= (x(t), y(t), \lambda(t))-(x^\star, y^\star, \lambda^\star)$. 
By definition of $v(t)$, we have
\[
v(t)= \sigma(t) z(t) + \delta(t) \dot{z}(t).
\]
From the triangle inequality and the bound \eqref{basic_maj_3}, we get
\[
\delta(t)\anorm{\dot{z}(t)} \leq  C \pa{1+\frac{\sigma(t)}{\sqrt{\xi(t)}}}.
\]
According to the definition \eqref{basic_choice_0} of $\delta (t)$ and $\xi(t)$, we get
\[
\anorm{(\dot {x}(t), \dot {y}(t), \dot{\lambda}(t))} \leq \frac{C}{\alpha(t)\sigma(t)}\pa{1+\sqrt{\frac{\sigma(t)}{\sigma (t)\pa{\gamma(t)\alpha (t)-\dot \alpha (t)-1} - 2 \alpha(t)  \dot \sigma(t)}}},
\]
which ends the proof. \qed 
\end{proof}

\subsection{The role of $\alpha$ and time scaling}
The time scaling parameter $b$ enters the conditions on the parameters only via \ref{cond:G4}, which therefore plays a central role in our analysis. Now consider relaxing \ref{cond:G4} to the inequality
\begin{enumerate}[label=($\cG_{\arabic*}^{+}$),start=4,itemindent=30ex]
\item $\frac{d}{dt}\left(\alpha^2 \sigma^2 b\right)(t) - \alpha (t)\sigma (t)^2 b(t) \geq 0$. \label{cond:G4+}
\end{enumerate}
This is a weaker assumption in which case the corresponding term in \eqref{basic-Liap-22} does not vanish.
However, such an inequality can still be integrated to yield meaningful convergence rates. This is what we are about to prove.
\begin{theorem}\label{Lyap_gen}
Suppose the assumptions of Theorem~\ref{ACFR,rescale}\ref{theoACFRrescale:itemb} hold, where condition \ref{cond:G4} is replaced with \ref{cond:G4+}. Let $(x^\star,y^\star,\lambda^\star) \in \sS \neq \emptyset$. Assume also that $\inf F(x,y) > -\infty$. Then, for all $t \geq t_0$
\begin{align}
\cL(x(t),y(t),\lambda^\star) - \cL(x^\star,y^\star,\lambda^\star) &= \cO\pa{\exp\pa{-\int_{t_0}^t\frac{1}{\alpha(s)}ds}}, \label{rat-conv-L*-gen} \\
\norm{Ax(t)+By(t)-c}^2 &= \cO\pa{\exp\pa{-\int_{t_0}^t\frac{1}{\alpha(s)}ds}} , \label{rat-conv-L*-gen-b} \\
-C_1\exp\pa{-\int_{t_0}^t\frac{1}{2\alpha(s)}ds} \leq F( x(t),y(t))-F^\star &\leq C_2\exp\pa{-\int_{t_0}^t\frac{1}{\alpha(s)}ds} , \label{rat-conv-L*-gen-obj}
\end{align}
where $C_1$ and $C_2$ are positive constants.
\end{theorem}

\begin{proof}
We embark from \eqref{basic-Liap-22} in the proof of Theorem~\ref{ACFR,rescale}. In view of \ref{cond:G4+} and \eqref{basic_choice_0}, \eqref{basic-Liap-22} becomes
\begin{equation}\label{eq:Eode}
0\geq \dfrac{d}{dt}\cE - \pa{\frac{d}{dt}\pa{\alpha^2 \sigma^2 b}-\alpha \sigma ^2 b} \cF_\mu(w) 
\geq \dfrac{d}{dt}\cE - \dfrac{\frac{d}{dt}\pa{\alpha^2 \sigma^2 b}-\alpha \sigma ^2 b}{\alpha^2 \sigma^2 b}\cE .
\end{equation}
Since $(x^\star,y^\star,\lambda^\star) \in \sS$, $\cF_\mu$ is non-negative and so is the Lyapunov function $\cE$. Integrating \eqref{eq:Eode}, we obtain the existence of a positive constant $C$ such that, for all $t\geq t_0$
\[
0 \leq \cE(t) \leq C \alpha(t)^2 \sigma (t)^2 b(t) \exp\pa{-\int_{t_0}^t\frac{1}{\alpha(s)}ds},
\]
which entails, after dropping the positive terms in $\cE$,
\begin{equation}\label{eq:Fmurate}
\cF_\mu(w(t))\leq C \exp\pa{-\int_{t_0}^t\frac{1}{\alpha(s)}ds}.
\end{equation}
\eqref{rat-conv-L*-gen} and \eqref{rat-conv-L*-gen-b} follow immediately from \eqref{eq:Fmurate} and the definition of $\cF_\mu$. 

Let us now turn to \eqref{rat-conv-L*-gen-obj}. Arguing as in the proof of Theorem~\ref{ACFR,rescale}\ref{theoACFRrescale:itemb}, we have
\begin{align*}
F( x(t),y(t))-F^\star \geq -\anorm{\lambda^\star}\anorm{Ax(t)+By(t)-c} .
\end{align*}
Plugging \eqref{rat-conv-L*-gen-b} in this inequality yields the lower-bound of \eqref{rat-conv-L*-gen-obj}. 

For the upper-bound, we will argue as in the proof of Theorem~\ref{ACFR,rescale}\ref{theoACFRrescale:itema} by considering $\lambda^\star$ as a free variable in $\cZ$. By assumption, we have $F$ is bounded from below. This together with \eqref{rat-conv-L*-gen-b} implies that $\cE$ is also bounded from below, and we denote $\underline{\cE}$ this lower-bound. Define $\tilde{\cE}(t) = \cE(t) - \underline{\cE}$ if $\underline{\cE}$ is negative and $\tilde{\cE}(t) = \cE(t)$ otherwise. Thus, from \eqref{eq:Eode}, it is easy to see that $\tilde{\cE}$ verifies
\begin{equation}\label{eq:Etode}
\dfrac{d}{dt}\tilde{\cE} \leq \dfrac{\frac{d}{dt}\pa{\alpha^2 \sigma^2 b}-\alpha \sigma ^2 b}{\alpha^2 \sigma^2 b}\tilde{\cE} .
\end{equation}
Integrating \eqref{eq:Etode} and arguing with the sign of $\underline{\cE}$, we get the existence of a positive constant $C$ such that, for all $t\geq t_0$
\[
\cE(t) \leq \tilde{\cE}(t) \leq C \alpha(t)^2 \sigma (t)^2 b(t) \exp\pa{-\int_{t_0}^t\frac{1}{\alpha(s)}ds} .
\]
Dropping the quadratic terms in $\cE$, this yields
\begin{align*}
F(x(t),y(t))-F^\star + \dotp{\lambda^\star}{Ax(t)+By(t)-c} \leq C \exp\pa{-\int_{t_0}^t\frac{1}{\alpha(s)}ds} .
\end{align*}
When $Ax(t)+By(t)-c = 0$, we are done by taking, \eg $\lambda^\star = 0$. Assume now that $Ax(t)+By(t)-c \neq 0$ and choose
\[
\lambda^\star =  \frac{Ax(t)+By(t)-c}{\anorm{Ax(t)+By(t)-c}} .
\]
We arrive at
\begin{align*}
F(x(t),y(t))-F^\star \leq F(x(t),y(t))-F^\star + \anorm{Ax(t)+By(t)-c} \leq C \exp\pa{-\int_{t_0}^t\frac{1}{\alpha(s)}ds} ,
\end{align*}
which completes the proof.
\qed
\end{proof}

\begin{remark}\label{rem:Lyap_gen}
Though the rates in Theorem~\ref{ACFR,rescale} and Theorem~\ref{Lyap_gen} look apparently different, it turns out that as expected, those of Theorem~\ref{ACFR,rescale} are actually a specialisation of those in Theorem~\ref{Lyap_gen} when \ref{cond:G4+} holds as an equality, \ie \ref{cond:G4} is verified.
To see this, it is sufficient to realize that, with the notation $a(t) \eqdef \alpha(t)^2 \sigma(t)^2 b(t)$, \ref{cond:G4} is equivalent to $\dot{a}(t) = \frac{1}{\alpha (t)}a(t)$. Upon integration, we obtain $a(t) = \exp\left(\int_{t_0}^t\frac{1}{\alpha(s)} ds\right) $, or equivalently
\[
\frac{1}{\alpha(t)^2 \sigma(t)^2 b(t)}= \exp\left(-\int_{t_0}^t\frac{1}{\alpha(s)} ds\right).
\]
\end{remark}

\section{Well-posedness of \eqref{eq:trials}}\label{Cauchy-problem}

In this section, we will show existence and uniqueness of a strong global solution to the Cauchy problem associated with~\eqref{eq:trials}. The main idea is to formulate \eqref{eq:trials} in the phase space as a non-autonomous first-order system. In the  smooth case, we will invoke the non-autonomous Cauchy-Lipschitz theorem \cite[Proposition~6.2.1]{haraux91}. In the non-smooth case, we will use a standard Moreau-Yosida smoothing argument.

\if
{
\begin{lemma}\label{haraux}\mbox{\rm (\cite[Prop. 6.2.1]{haraux91})} 
Let  $G:I\times\mathcal Z\rightarrow \mathcal Z$ where  $I=[t_0,+\infty[$ and $\mathcal Z$ is a Banach space. Assume that 

\smallskip

\, (i) for every $z\in \mathcal Z$, $G(\cdot,z)\in L^1_{loc}(I, \mathcal Z)$;

\smallskip

\, (ii)  for a.e. $t\in I$, for every $z_1, \, z_2\in  \mathcal Z$,
$$
\| G(t,z_1)-G(t,z_2)\|\leq K(t,\|z_1\|+\|z_2\|)\|z_1-z_2\|, \text{ where } K(\cdot, r)\in L^1_{loc}(I), \forall r\in\mathbb R_+;
$$

\, (iii) for a.e. $t\in I$,  for every $z\in  \mathcal Z$,
$$
\| G(t,z)\|\leq P(t)(1+\| z \|),\text{ where } P\in L^1_{loc}(I).
$$
Then, for every $s\in I, z\in  \mathcal Z$, there exists a unique solution $u_{s,z}\in W^{1,1}_{loc}(I,\mathcal Z)$ of the Cauchy problem:
\begin{center}
$\dot u_{s,z}(t)=G(t,u_{s,z}(t))$ for a.e. $t\in I$, and $u_{s,z}(s)=z$.
\end{center}
\end{lemma}

}
\fi

\subsection{Case of globally Lipschitz continuous gradients}
We consider first the case where the gradients of $f$ and $g$ are globally Lipschitz continuous over $\cX$ and $\cY$. Let us start by recalling the notion of strong solution.
\begin{definition}\label{def:strongsol}
Denote $\cH \eqdef \cX \times \cY \times \cZ$ equipped with the corresponding product space structure, and $w: t \in [t_0,+\infty[ \mapsto (x(t),y(t),\lambda(t)) \in \cH$. The function $w$ is a strong global solution of the dynamical system \eqref{eq:trials} if it satisfies the following properties:
\begin{enumerate}[label=$\bullet$]
\item $w$ is in $\cC^1([0,+\infty[;\cH)$;
\item $w$ and $\dot w$ are absolutely continuous on every compact subset of the interior of $[t_0,+\infty[$ (hence almost everywhere differentiable);
\item for almost all $t \in [t_0,+\infty[$, \eqref{eq:trials} holds with $w(t_0) = (x_0,y_0,\lambda_0)$ and $\dot w(t_0) = (u_0,v_0,\nu_0))$.
\end{enumerate}
\end{definition}

\begin{theorem}\label{thm:wellglobal}
Suppose that \eqref{eq:HP} holds\footnote{Actually, convexity is not needed here.} and, moreover, that $\nabla f$ and $\nabla g$ are Lipschitz continuous, respectively over $\cX$ and $\cY$. Assume  that $\gamma, \,  \alpha,  \, b: [t_0, +\infty[ \to \R^+$  are non-negative  continuous functions. Then, for any given initial condition $(x(t_0), \dot{x}(t_0))=(x_0,\dot x_0)\in \cX \times \cX$,  $(y(t_0), \dot{y}(t_0))=(y_0,\dot y_0)\in \cY \times \cY$, $(\lambda(t_0), \dot{\lambda}(t_0))=(\lambda_0,\dot \lambda_0)\in \cZ \times \cZ$, the evolution system \eqref{eq:trials} has a unique strong global solution.
\end{theorem}

\begin{proof} 
Recall the notations of Definition~\ref{def:strongsol}. Let $I = [t_0,+\infty[$ and let $Z: t \in I \mapsto (w(t),\dot w(t)) \in \cH^2$. \eqref{eq:trials} can be equivalently written as the Cauchy problem on $\cH^2$
\begin{equation}\label{syst1g}
\begin{cases}
\dot Z(t) + G(t,Z(t)) = 0 & \text{ for } t \in I, \\
Z(t_0)=Z_0 ,
\end{cases}
\end{equation}
where $Z_0=(x_0,y_0,\lambda_0,u_0,v_0,\nu_0)$, and $G: I \times \cH^2 \to \cH^2$ is the operator
\begin{equation}\label{def:G}
G(t, (x,y,\lambda),(u,v,\nu))= 
\begin{pmatrix}
-u \\
-v \\
-\nu  \\
\gamma(t)u + b(t)\bpa{\nabla f(x) + A^\top\pa{\lambda+\alpha(t)\nu + \mu (Ax+By-c)}} \\
\gamma(t)v + b(t)\bpa{\nabla g (y) + B^\top\pa{\lambda+\alpha(t)\nu + \mu (Ax+By-c)}} \\
\gamma(t)\nu - b(t)\bpa{A(x + \alpha(t)u)) + B(y + \alpha(t)v) - c}
\end{pmatrix} .
\end{equation}
To invoke \cite[Proposition~6.2.1]{haraux91}, it is sufficient to check that for a.e. $t \in I$, $G(t,\cdot)$ is $\beta(t)$-Lipschitz continuous with $\beta(\cdot) \in L^1_{loc}(I)$, and for a.e. $t \in I$, $G(t,Z) = \cO(P(t)(1+\anorm{Z})$, $\forall Z \in \cH^2$, with $P(\cdot) \in L^1_{loc}(I)$. Since $\nabla f$, $\nabla g$ are globally Lipschitz continuous, and $A$ and $B$ are bounded linear, elementary computation shows that there exists a constant $C > 0$ such that
\[
\anorm{G(t,Z) - G(t,\bar Z)} \leq C \beta(t) \anorm{Z - \bar{Z}} , \quad \beta(t) = 1+\gamma(t)+b(t)(1+\alpha(t)) .
\]
Owing to the continuity of the parameters $\gamma (\cdot)$, $\alpha (\cdot)$, $b(\cdot)$, $\beta(\cdot)$ is integrable on $[t_0,T]$ for all $t_0 < T < +\infty$. Similar calculation shows that
\[
\anorm{G(t,(x,y,\lambda),(u,v,\nu))} \leq C \beta(t) \bpa{\anorm{\pa{\nabla f(u),\nabla g(v)}}+\anorm{\pa{x,y,\lambda,u,v,\nu}}} ,
\]
and we conclude similarly. It then follows from \cite[Proposition~6.2.1]{haraux91} that there exists a unique global solution $Z(\cdot) \in W^{1,1}_{loc}(I;\cH^2)$ of \eqref{syst1g} satisfying the initial condition $Z(t_0)=Z_0$, and thus, by \cite[Corollary~A.2]{Bre1} that $Z(\cdot)$ is a strong global solution to \eqref{syst1g}. This in turn leads to the existence and uniqueness of a strong solution $(x(\cdot),y(\cdot),\lambda(\cdot))$ of \eqref{eq:trials}. \qed
\end{proof}

\begin{remark} 
One sees from the proof that for the above result to hold, it is only sufficient to assume that the parameters $\gamma$, $\alpha$, $b$ are locally integrable instead of continuous. In addition, in the above results, we even have existence and uniqueness of a classical solution.
\end{remark}


\subsection{Case of locally Lipschitz continuous gradients}
Under local Lipschitz continuity assumptions on the gradients $\nabla f$ and  $\nabla g$, the operator $Z \mapsto G(t,Z)$ defined in \eqref{def:G} is only Lipschitz continuous over the bounded subsets of $\cH^2$.
As a consequence, the Cauchy-Lipschitz theorem  provides the existence and uniqueness of a local solution. To pass from a local solution to a global solution, we will rely on the estimates established in Theorem~\ref{ACFR_rescale_boundedness}.

\begin{theorem}\label{thm:welllocal}
Suppose that \eqref{eq:HP} holds\footnote{Again, convexity is superfluous here.} and, moreover, that $\nabla f$ and $\nabla g$ are Lipschitz continuous over the bounded subsets of respectively $\cX$ and $\cY$. Assume  that $\gamma, \,  \alpha,  \, b: [t_0, +\infty[ \to \R^+$  are non-negative continuous functions such that the conditions \ref{cond:G1+}, \ref{cond:G2}, \ref{cond:G3}, \ref{cond:G4} and \ref{cond:G5} are satisfied, and that $\sup_{t \geq t_0} \sigma(t) < +\infty$. Then, for any initial condition $(x(t_0), \dot{x}(t_0))=(x_0,\dot x_0)\in \cX \times \cX$,  $(y(t_0), \dot{y}(t_0))=(y_0,\dot y_0)\in \cY \times \cY$, $(\lambda(t_0), \dot{\lambda}(t_0))=(\lambda_0,\dot \lambda_0)\in \cZ \times \cZ$, the evolution system \eqref{eq:trials} has a unique strong global solution.
\end{theorem}

\begin{proof}
We use the same notation as in the proof of Theorem~\ref{thm:wellglobal}. Let us consider the maximal solution of the Cauchy problem \eqref{syst1g}, say $Z : [t_0, T[ \to \cH^2$. We have to prove that $T=+\infty$.
Following a classical argument, we argue by contradiction, and suppose that $T < +\infty$. It is then sufficient to prove that the limit of $Z(t)$ exists as $t \to T$, so that it will be possible to extend $Z$ locally to the right of $T$ thus getting a contradiction. According to the Cauchy criterion, and the constitutive equation $\dot Z(t) = G(t,Z(t))$, it is sufficient to prove that $Z(t)$ is bounded over $[t_0,T[$. At this point, we use the estimates provided by Theorem~\ref{ACFR_rescale_boundedness}, which gives precisely this result under the conditions imposed on the parameters. \qed
\end{proof}

\subsection{The non-smooth case}
For a large number of applications (\eg data processing, machine learning, statistics), non-smooth functions are ubiquitous. To cover these practical situations, we need to consider the case where the functions $f$ and $g$ are non-smooth. In order to adapt the dynamic \eqref{eq:trials} to this non-smooth situation, we will consider the corresponding differential inclusion
\begin{equation}
\begin{cases}
\ddot x+\gamma (t) \dot x + b(t) \bpa{\partial f(x) +
A^\top \brac{\lambda + \alpha(t) \dot\lambda 
+ \mu (Ax+By-c)}} &\ni 0 \\
\ddot y+\gamma (t)\dot y + b(t)\bpa{\partial g(y) +
B^\top \brac{\lambda + \alpha(t) \dot\lambda 
+ \mu (Ax+By-c)}} &\ni 0 \\
\ddot \lambda+\gamma (t)\dot \lambda - b(t)
\bpa{A(x + \alpha(t)\dot x) + B(y + \alpha(t)\dot y) -c}  &= 0\\
(x(t_0),y(t_0),\lambda(t_0)) = (x_0,y_0,\lambda_0) \qandq \\
(\dot x(t_0),\dot y(t_0),\dot \lambda(t_0)) = (u_0,v_0,\nu_0) ,
\end{cases}\label{eq:trialsnonsmooth}
\end{equation}
where $\partial f$ and $\partial g$ are the subdifferentials of $f$ and $g$, respectively. Beyond global existence issues that we will address shortly, one may wonder whether our Lyapunov analysis in the previous sections is still valid in this case. The answer is affirmative provided one takes some care in two main steps that are central in our analysis. First, when taking the time-derivative of the Lyapunov, one has to invoke now the (generalized) chain rule for derivatives over curves (see \cite{Bre1}). The second ingredient is the validity of the subdifferential inequality for convex functions. In turn all our results and estimates presented in the previous sections can be transposed to this more general non-smooth context. Indeed, our approximation scheme that  we will present shortly turns out to be monotonically increasing. This  gives a variational convergence (epi-convergence) which allows to simply pass to the limit over the estimates established in the smooth case.

Let us now turn to the existence of a global solution to \eqref{eq:trialsnonsmooth}. We will again consider strong solutions to this problem, \ie solutions that are $\cC^1([t_0,+\infty[;\cH)$, locally absolutely continuous, and \eqref{eq:trialsnonsmooth} holds almost everywhere on $[t_0,+\infty[$. A natural idea is to use the Moreau-Yosida regularization in order to bring the problem to the smooth case before passing to an appropriate limit. Recall that, for any $\theta > 0$, the Moreau envelopes $f_{\theta}$ and $g_{\theta}$ of $f$ and $g$ are defined respectively by
\[
f_{\theta} (x)= \min_{\xi \in \cX} \bra{f(\xi) + \frac{1}{2 \theta} \anorm{x-\xi}^2}, \quad 
g_{\theta} (y)= \min_{\eta \in\cY} \bra{g(\eta)+ \frac{1}{2 \theta} \anorm{y-\eta}^2}.
\]
As a classical result, $f_{\theta}$ and $g_{\theta}$ are continuously differentiable and their gradients are $\frac{1}{\theta}$-Lipschitz continuous. We are then led to consider, for each $\theta >0$, the dynamical system
\begin{equation}
\begin{cases}
\ddot x_\theta +\gamma (t) \dot x_\theta + b(t) \bpa{\nabla f_{\theta} (x_\theta) +
A^\top \brac{\lambda_\theta + \alpha(t) \dot\lambda_\theta 
+ \mu (Ax_\theta+By_\theta-c)}} &=0 \\
\ddot y_\theta+\gamma (t)\dot y_\theta + b(t)\bpa{\nabla g_{\theta} (y_\theta) +
B^\top \brac{\lambda_\theta + \alpha(t) \dot\lambda_\theta 
+ \mu (Ax_\theta+By_\theta-c)}}  &=0 \\
\ddot \lambda_\theta+\gamma (t)\dot \lambda_\theta - b(t)
\bpa{A(x_\theta + \alpha(t)\dot x_\theta) + B(y_\theta + \alpha(t)\dot y_\theta) -c} &= 0 .
\end{cases}\label{eq:trialssmoothed}
\end{equation}
The system \eqref{eq:trialssmoothed} comes under our previous study, and for which we have existence and uniqueness of a strong global solution. In doing so, we generate a filtered sequence $(x_{\theta},y_{\theta},\lambda_{\theta})_{\theta}$ of trajectories.\\

The challenging question is now to pass to the limit in the system above as $ \theta \to 0^+$.
This is a non-trivial problem, and to answer it, we have to assume that the spaces $\cX$, $\cY$ and $\cZ$ are finite dimensional, and that $f$ and $g$ are convex real-valued (\ie $\dom(f) = \cX$ and $\dom(g)=\cY$ in which case $f$ and $g$ are continuous). Recall that $\partial F(x,y) = \partial f(x) \times \partial g(y)$, and denote $\brac{\partial F(x,y)}^0$ the minimal norm selection of $\partial F(x,y)$.

\begin{theorem}\label{thm:wellnonsmooth}
Suppose that $\cX$, $\cY$, $\cZ$ are finite dimensional Hilbert spaces, and that the functions $f: \cX \to \R $ and  $g: \cY \to \R$ are convex. Assume that 
\begin{enumerate}[label=(\roman*)]
\item $F$ is coercive on the affine feasibility set; \label{cond:coer}
\item $\beta_F \eqdef \sup_{(x,y) \in \cX \times \cY} \anorm{\brac{\partial F(x,y)}^0}< +\infty$; \label{cond:subdiffbnd}
\item the linear operator $L = [A ~ B]$ is surjective. \label{cond:Lsurj}
\end{enumerate}
Suppose also that $\gamma, \,  \alpha,  \, b: [t_0, +\infty[ \to \R^+$ are non-negative continuous functions such that the conditions \ref{cond:G1+}, \ref{cond:G2}, \ref{cond:G3} \ref{cond:G4} and \ref{cond:G5} are satisfied, and that $\sup_{t \geq t_0} \sigma(t) < +\infty$. Then, for any initial condition $(x(t_0), \dot{x}(t_0))=(x_0,\dot x_0)\in \cX \times \cX$,  $(y(t_0), \dot{y}(t_0))=(y_0,\dot y_0)\in \cY \times \cY$, $(\lambda(t_0), \dot{\lambda}(t_0))=(\lambda_0,\dot \lambda_0)\in \cZ \times \cZ$, the evolution system \eqref{eq:trialsnonsmooth} admits a strong global solution .
\end{theorem}

Condition~\ref{cond:coer} is natural and ensures for instance that the solution set of \eqref{eq:P} is non-empty. Condition~\ref{cond:Lsurj} is also very mild. A simple case where \ref{cond:subdiffbnd} holds is when $f$ and $g$ are Lipschitz continuous. 

\begin{proof}
The key property is that the estimates obtained in Theorem~\ref{ACFR,rescale} and \ref{ACFR_rescale_boundedness}, when applied to \eqref{eq:trialssmoothed}, have a favorable dependence on $\theta$. Indeed, a careful examination of the estimates shows that $\theta$ enters them through the Lyapunov function at $t_0$ only via $|F_{\theta}(x_0,y_0) - F_{\theta}(x_{\theta}^\star,y_{\theta}^\star)|$ and $\anorm{(x_{\theta}^\star,y_{\theta}^\star,\lambda_{\theta}^\star)}$, where $F_{\theta}(x,y)=f_{\theta}(x)+g_{\theta}(y)$, $(x_{\theta}^\star,y_{\theta}^\star) \in \argmin_{Ax+Bx=c} F_{\theta}(x,y)$ and $\lambda_{\theta}^\star$ is an associated dual multiplier; see \eqref{eq:C0}. With standard properties of the Moreau envelope, see \cite[Chapter~3]{AttouchBook} and \cite[Chapter~12]{BauschkeCombettes}, one can show that for all $(x,y) \in \cX \times \cY$
\[
F(x,y) - \frac{\theta}{2}\anorm{\brac{\partial F(x,y)}^0}^2 \leq F_{\theta}(x,y) \leq F(x,y) .
\]
This, together with the fact that $(x_{\theta}^\star,y_{\theta}^\star) \in \argmin_{Ax+Bx=c} F_{\theta}(x,y)$ and $(x^\star,y^\star) \in \argmin_{Ax+Bx=c} F(x,y)$ yields
\[
F_{\theta}(x_{\theta}^\star,y_{\theta}^\star) \leq F_{\theta}(x^\star,y^\star) \leq F(x^\star,y^\star) \leq F(x_{\theta}^\star,y_{\theta}^\star) .
\]
Thus
\begin{multline*}
F(x_0,y_0) - F(x^\star,y^\star) - \frac{\theta}{2}\anorm{\brac{\partial F(x_0,y_0)}^0}^2 \leq F_{\theta}(x_0,y_0) - F_{\theta}(x_{\theta}^\star,y_{\theta}^\star) \\
\leq F(x_0,y_0) - F(x^\star,y^\star) + \frac{\theta}{2}\anorm{\brac{\partial F(x_{\theta}^\star,y_{\theta}^\star)}^0}^2 .
\end{multline*}
This entails, owing to \ref{cond:subdiffbnd}, that
\[
\abs{F_{\theta}(x_0,y_0) - F_{\theta}(x_{\theta}^\star,y_{\theta}^\star)} \leq \abs{F(x_0,y_0) - F(x^\star,y^\star)} + \frac{\beta_F^2\theta}{2}
\]
and thus, since we are interested in the limit as $\theta \to 0^+$, 
\[
\sup_{\theta \in [0,\bar{\theta}]} \abs{F_{\theta}(x_0,y_0) - F_{\theta}(x_{\theta}^\star,y_{\theta}^\star)} \leq \abs{F(x_0,y_0) - F(x^\star,y^\star)} + \frac{\beta_F^2\bar{\theta}}{2} < +\infty .
\] 
On the other hand, 
\[
F(x_{\theta}^\star,y_{\theta}^\star) \leq F_{\theta}(x_{\theta}^\star,y_{\theta}^\star) + \frac{\beta_F^2\theta}{2} \leq F_{\theta}(x^\star,y^\star) + \frac{\beta_F^2\theta}{2} \leq F(x^\star,y^\star) + \frac{\beta_F^2\bar{\theta}}{2} .
\]
Thus, in view of \ref{cond:coer}, $\exists a > 0$ and $b \in \R$ such  that
\[
a\anorm{(x_{\theta}^\star,y_{\theta}^\star)} + b \leq F(x^\star,y^\star) + \frac{\beta_F^2\bar{\theta}}{2} ,
\]
which shows that 
\[
\sup_{\theta \in [0,\bar{\theta}]} \anorm{(x_{\theta}^\star,y_{\theta}^\star)} < +\infty .
\]
Let us turn to $\lambda_{\theta}^\star$. When $\lambda_{\theta}^\star$ is chosen as in \eqref{eq:lambdaunit}, then we are done. When $\lambda_{\theta}^\star$ is the optimal dual multiplier satisfying \eqref{opt_system}, then it is a solution to the Fenchel-Rockafellar dual problem
\[
\min_{\lambda \in \cZ} F_\theta^*(-L^*\lambda) + \dotp{c}{\lambda} ,
\]
where $F_\theta^*$ is the Legendre-Fenchel conjugate of $F_\theta$. Without loss of generality, we assume $c = 0$. Classical conjugacy results give
\[
F_\theta^*(u) = F^*(u) + \frac{\theta}{2}\norm{u}^2 .
\]
Since $f$ and $g$ are convex and real-valued, the domain of $F$ is full. This is equivalent to coercivity of $F^*$. 
This together with injectivity of $L^*$ (see \ref{cond:Lsurj}), imply that there exists $a > 0$ and $b \in \R$ (potentially different from those above) such that
\[
a\anorm{\lambda_{\theta}^\star} + b \leq F^*(-L^*\lambda_{\theta}^\star) \leq F_{\theta}^*(-L^*\lambda_{\theta}^\star) \leq F_{\theta}^*(-L^*\lambda^\star) \leq F^*(-L^*\lambda^\star) + \frac{\bar{\theta}}{2}\norm{L^*\lambda^\star}^2 < +\infty.
\]
Altogether this shows that
\[
\sup_{\theta \in [0,\bar{\theta}]} \anorm{\lambda_{\theta}^\star} < +\infty .
\]
Combining the above with Theorem~\ref{ACFR_rescale_boundedness}, we conclude that for all $T > t_0$, the trajectories $(x_{\theta}(.),y_{\theta}(.),\lambda_{\theta}(\cdot))$ and the velocities $(\dot x_{\theta}(.),\dot y_{\theta}(.),\dot \lambda_{\theta}(\cdot))$ are bounded in $L^2(t_0,T;\cX \times \cY)$ uniformly in $\theta$. Since $\cX$ and $\cY$ are finite dimensional spaces, we deduce by the Ascoli-Arzel\`{a} theorem, that the trajectories are relatively compact for the uniform convergence over the bounded time intervals. By properties of the Moreau envelope, we also have, for all $(x,y) \in \cX \times \cY$,
\[
\anorm{\nabla F_{\theta}(x,y)} \nearrow \anorm{\brac{\partial F(x,y)}^0} \text{ as } \theta \searrow 0,
\]
and thus
\[
\anorm{\nabla F_{\theta}(x,y)} \leq \beta_F .
\]
Using this and the boundedness assertions of the trajectories and velocities proved above in the constitutive equations \eqref{eq:trialssmoothed}, the acceleration remains also bounded on the bounded time intervals. Passing to the limit as $\theta \to 0^+$ in \eqref{eq:trialssmoothed} is therefore relevant by a classical maximal monotonicity argument. Indeed, we work with the canonical extension of the maximally monotone operators $\nabla F_\theta$ and $\partial F$ to $L^2(t_0,T, \cX \times \cY)$, and, in this functional setting, we use that $\nabla F_\theta$ graph converges to $\partial F$ in the strong-weak topology. \qed
\end{proof}

We conclude this section by noting that at this stage, uniqueness of the solution to \eqref{eq:trialsnonsmooth} is a difficult open problem. In fact, even existence in infinite dimension and/or with any proper lower semicontinuous convex functions $f$ and $g$ is not clear. This goes far beyond the scope of the present paper and we leave it to a future work.  
 
\if
{
Differentiating $\cE$ with respect to $t$ gives
\begin{equation}\label{der-E-H}
\dfrac{d}{dt}\cE (t)=\dot{\delta}(t) \cF_\mu(w(t))+ \delta (t) \dotp{\nabla \cF_\mu(w(t))}{\dot{w}(t)}+  \left\langle v(t),\dot{v}(t) \right\rangle .
\end{equation}
By definition of $v(t)$, and by using the constitutive equation  \eqref{eq:trials}, we have
 $$\begin{array}{lll}
\dot{v}(t) & = & \gamma_0 \dot w (t) + \beta(t)  \nabla_{\alpha}\cL_\mu(w(t)) + t \Big( \ddot w (t) + \beta(t) \dfrac{d}{dt} \nabla_{\alpha}\cL_\mu(w(t)) + \dot{\beta}(t) \nabla_{\alpha}\cL_\mu(w(t)) \Big)  \\
		&=& \gamma_0 \dot w (t) + \beta(t)  \nabla_{\alpha}\cL_\mu(w(t)) + t \Big( -\frac{\gamma_0}{t}\dot w (t) -b(t)\nabla_{\alpha}\cL_\mu(w(t)) + \dot{\beta}(t) \nabla_{\alpha}\cL_\mu(w(t)) \Big)\\
&=& t\left(\dot{\beta}(t) + \frac{\beta (t)}{t}  -b(t)  \right) \nabla_{\alpha}\cL_\mu(w(t)) .		
\end{array} $$
Elementary computation gives
$$
\nabla_\alpha\cL_\mu(w(t)) = \nabla \cF_\mu(w(t)) + C(t) 
$$
where
\begin{eqnarray*}
C(t)\eqdef\left[
\begin{array}{l} 
A^\top(\lambda (t)-\lambda^\star+\alpha (t) \dot\lambda (t))\vspace{2mm}\\  
B^\top(\lambda (t)-\lambda^\star+\alpha (t)\dot\lambda (t) )\vspace{2mm} \\  
-A(x(t)+\alpha (t)\dot x (t))-B(y (t)+\alpha (t)\dot y (t))+c
\end{array}
\right].
\end{eqnarray*}
Unambiguously, to shorten formulas, we sometimes omit the variable $t$.
According to the above formulas for $ v(t) $ and $ \dot v(t) $, we get
\begin{eqnarray*}
\dotp{v(t)}{\dot{v(t)}} 
	& =& t\left(\dot{\beta} + \frac{\beta}{t}  -b  \right)  \big\langle \nabla_{\alpha}\cL_\mu(w),  (\gamma_0 -1)( w-w^\star)+t\Big( \dot w  + \beta  \nabla_{\alpha}\cL_\mu(w) \Big)   \big\rangle\\
	&=&(\gamma_0 -1) t\left(\dot{\beta} + \frac{\beta}{t}  -b  \right)  \big\langle \nabla_{\alpha}\cL_\mu(w),   w-w^\star  \big\rangle\\
	& + &  t^2\left(\dot{\beta} + \frac{\beta}{t}  -b  \right)  \big\langle \nabla_{\alpha}\cL_\mu(w),   \dot w    \big\rangle \\
&+&	t^2\left(\dot{\beta} + \frac{\beta}{t}  -b  \right)  \beta 
\| \nabla_{\alpha}\cL_\mu(w) \|^2.
\end{eqnarray*}
 \tcb{Let us} insert this expression in \eqref{der-E-H}.  We obtain
\begin{eqnarray*}
\dfrac{d}{dt}\cE (t)&=& \dot{\delta}(t) \cF_\mu(w(t))+ \delta (t) \dotp{\nabla_\alpha\cL_\mu(w))}{\dot{w}(t)}-\delta (t)\dotp{C(t)}{\dot{w}(t)} \\
	&+&(\gamma_0 -1) t\left(\dot{\beta} + \frac{\beta}{t}  -b  \right)  \big\langle \nabla_{\alpha}\cL_\mu(w),   w-w^\star  \big\rangle\\
	& + &  t^2\left(\dot{\beta} + \frac{\beta}{t}  -b  \right)  \big\langle \nabla_{\alpha}\cL_\mu(w),   \dot w    \big\rangle 
+	t^2\left(\dot{\beta} + \frac{\beta}{t}  -b  \right)  \beta 
\| \nabla_{\alpha}\cL_\mu(w) \|^2.
\end{eqnarray*} 
Take
$$
\delta (t) = t^2 \left(b(t) -\dot{\beta}(t) - \frac{\beta (t)}{t} \right) 
$$ 
so that the term $\big\langle\nabla\cF_\mu(w ) , \dot w \big\rangle $  occurs with its opposite, and therefore disappears. 
Moreover we suppose that $\delta (t)$ is nonnegative.
 Thus, the above formula  simplifies to
\begin{eqnarray*}
&&\dfrac{d}{dt}\cE (t) +  \beta (t)\delta (t)
\| \nabla_{\alpha}\cL_\mu(w) \|^2 + (\gamma_0 -1) \frac{\delta (t)}{t} \big\langle \nabla_{\alpha}\cL_\mu(w (t)),   w(t)-w^\star  \big\rangle - \dot{\delta}(t) \cF_\mu(w(t))\\
&&\hspace{2cm} +\delta (t)\dotp{C(t)}{\dot{w}(t)} =0.
\end{eqnarray*} 
By convexity of $\cF_\mu$, we have
\[ 
\cL_\mu(w^\star)- \cL_\mu(w (t)) \geq\dotp{\nabla_{\alpha}\cL_\mu(w (t))}{w^\star-w(t)},
\]
which, by definition of $\cF_\mu$ gives
$$
\big\langle \nabla_{\alpha}\cL_\mu(w (t)),   w(t)-w^\star  \big\rangle \geq \cF_\mu (w(t)). 
$$
Therefore
\begin{eqnarray*}
&&\dfrac{d}{dt}\cE (t) +  \beta (t)\delta (t)
\| \nabla_{\alpha}\cL_\mu(w) \|^2 + \Big((\gamma_0 -1) \frac{\delta (t)}{t} - \dot{\delta}(t)\Big) \cF_\mu(w(t))+\delta (t)\dotp{C(t)}{\dot{w}(t)} =0.
\end{eqnarray*}

On the other hand, a similar computation as in Theorem 1 gives

\[
\begin{array}{lll} \dotp{C(t)}{\dot{w}(t)}&=&
	 \sigma \big\langle \lambda-\lambda^\star+\alpha \dot\lambda \, ,\, Ax-Ax^\star  \big\rangle + \delta   \big\langle\lambda-\lambda^\star+\alpha \dot\lambda   , A\dot x \big\rangle \vspace{1mm}\\
	&&+
	 \sigma \big\langle \lambda-\lambda^\star+\alpha \dot\lambda \, ,\, By-By^\star  \big\rangle + \delta   \big\langle\lambda-\lambda^\star+\alpha \dot\lambda   , B\dot y \big\rangle \vspace{1mm}\\
	&&- 
	 \sigma \langle A(x+\alpha \dot x)+B(y+\alpha \dot y)-c\, ,\, \lambda-\lambda^\star  \rangle \vspace{1mm} \\
	 &&- \delta  \langle A(x+\alpha \dot x)+B(y+\alpha \dot y)-c\,   , \dot \lambda \big\rangle .
\end{array} 
\]
\tcb{Let us} recall that  $(x^\star,y^\star) \in \cX\times\cK $ is a solution  of \eqref{eq:P}. Hence  $A(x^\star) + B(y^\star) =c$.  According to this property, \tcb{Let us} rearrange $\cW$ as follows:
\[
\begin{array}{lll}
\cW&=&
	 \sigma \big\langle \lambda-\lambda^\star+\alpha \dot\lambda \, ,\, Ax+By-c \big\rangle + \delta   \big\langle\lambda-\lambda^\star+\alpha \dot\lambda   , A\dot x + B\dot y\big\rangle \vspace{1mm} \\
	&&- 
	 \sigma \big\langle Ax+By-c\, ,\, \lambda-\lambda^\star  \big\rangle - 
	 \sigma\alpha \big\langle A \dot x+B \dot y\, ,\, \lambda-\lambda^\star  \big\rangle \vspace{1mm} \\
	 &&- \delta   \big\langle Ax+By-c\, ,\dot \lambda \big\rangle 
	 - \delta\alpha   \big\langle A \dot x+B \dot y\,   , \dot \lambda \big\rangle \vspace{1mm} \\
	 &=& (\sigma \alpha - \delta)\Big[ \big\langle Ax+By-c\, ,\dot \lambda \big\rangle -  \big\langle A \dot x+B \dot y\, ,\, \lambda-\lambda^\star  \big\rangle  \Big].
\end{array} 
\]
Since it is difficult to control the sign of the above expression, we are naturally led to make the choice: for all $t\geq t_0$
\begin{equation} \label{basic_choice}
\delta (t)= \sigma(t) \alpha (t), 
\end{equation}
which gives $\cW =0$. 

\tcb{Let us} return to \eqref{basic-Lyap1}. Collecting the above results,
we get
\begin{eqnarray*}
\dfrac{d}{dt}\cE +\left( \delta  b 
	\sigma-\dfrac{d}{dt}(\delta^2b) \right)\cF_\mu(w) 
	 \leq \big(\frac12\dot\xi +  \sigma\dot\sigma \big) \|w-w^\star\|^2  + [\sigma-\dot\delta-\delta\gamma]\delta \|\dot w\|^2.
\end{eqnarray*}
Thus, assuming the conditions 
\begin{eqnarray}
&& \sigma-\dot\delta-\delta\gamma \leq 0  \label{def:sigma2-H}
\vspace{2mm}\\
&&\sigma\dot\sigma+\frac12\dot\xi\leq 0  \label{def:sigma1-H}
\end{eqnarray}
 we obtain the inequality
\begin{equation}\label{basic-Liap-22-H}
\dfrac{d}{dt}\cE +\left( \delta  b 
	\sigma-\dfrac{d}{dt}(\delta^2b) \right)\cF_\mu(w) \leq 0.
\end{equation}
%
The sign of $\cF_\mu (w)(t)$ is a priori unknown, because of the term 
$\langle\lambda^\star , Ax(t)+By(t)-c\rangle  $ which comes within its definition.Therefore, we are led to assume that
\begin{equation}\label{def:sigma3-H}
\delta  b \sigma-\dfrac{d}{dt}(\delta^2b)= 0.
\end{equation}		
Then, inequality \eqref{basic-Liap-22-H} gives $\dfrac{d}{dt}\cE(t)\leq 0$ on $(t_0,+\infty)$. 
That's the basic ingredient of the Lyapunov analysis.
}
\fi

\section{The uniformly convex case}\label{sec:strongly_convex}
%
We now turn to examine the convergence properties of the trajectories generated by \eqref{eq:trials}, when the objective $F$ in \eqref{eq:P} is uniformly convex on bounded sets. Recall, see \eg \cite{BauschkeCombettes}, that $F: \cX \times \cY \to \R$ is uniformly convex on bounded sets if, for each $r > 0$, there is an increasing function $\psi_r: [0,+\infty[ \to [0,+\infty[$ vanishing only at the origin, such that 
\begin{equation}\label{eq:uniformconv}
F(v) \geq F(w) + \dotp{\nabla F(w)}{v-w} + \psi_r(\anorm{v-w})
\end{equation}
for all $(v,w) \in (\cX \times \cY)^2$ such that $\anorm{v} \leq r$ and $\anorm{w} \leq r$. The strongly convex case corresponds to $\psi_r(t) = c_F t^2/2$ for some $c_F > 0$. In finite dimension, strict convexity of $F$ entails uniform convexity on any non-empty bounded closed convex subset of $\cX \times \cY$, see \cite[Corollary~10.18]{BauschkeCombettes}.

\begin{theorem} \label{thm:convergence_s}
Suppose that $F$ is uniformly convex on bounded sets, and let $(x^\star, y^\star)$ be the unique solution of the minimization problem \eqref{eq:P}. Assume also that $\sS$, the set of saddle points of $\cL$ in \eqref{eq:minmax} is non-empty. Suppose that the conditions \ref{cond:G1+}--\ref{cond:G4} on the coefficients of \eqref{eq:trials} are satisfied for all $t \geq t_0$. Then, each solution trajectory  $t\in [t_0, +\infty[  \mapsto (x(t),y(t),\lambda(t))$ of \eqref{eq:trials} satisfies, $\forall t \geq t_0$,
\[
\psi_r\pa{\anorm{(x(t),y(t))-(x^\star,y^\star)}} = \cO \pa{ \frac{1}{\alpha(t)^2 \sigma(t)^2 b(t)}}.
\]
As a consequence, assuming that $\lim_{t\to +\infty} \alpha(t)^2 \sigma(t)^2 b(t) = +\infty$, we have that the trajectory $t\mapsto (x(t), y(t))$ converges strongly to $(x^\star, y^\star)$ as $t\to +\infty$.
\end{theorem}

\begin{proof}
Uniformly convex functions are strictly convex and coercive, and thus $(x^\star,y^\star)$ is unique. From Theorem~\ref{ACFR_rescale_boundedness}\ref{ACFR_rescale_boundedness:itemii}, there exists $r_1 > 0$ such that
\[
\sup_{t \geq t_0} \norm{(x(t),y(t))-(x^\star,y^\star)} \leq r_1 .
\]
Taking $r \geq r_1 + \norm{(x^\star,y^\star)}$, we have that the trajectory $(x(\cdot),y(\cdot))$ and $(x^\star,y^\star)$ are both contained in the ball of radius $r$ centered at the origin. Let $\lambda^\star$ be a Lagrange multiplier of problem \eqref{eq:P}, \ie $(x^\star,y^\star,\lambda^\star) \in \sS$. On the one hand, applying the uniform convexity inequality \eqref{eq:uniformconv} at $v=(x(t),y(t))$ and $w=(x^\star,y^\star)$, we have
\begin{multline*}
F(x(t),y(t)) \geq F(x^\star,y^\star) + \dotp{\nabla F(x^\star,y^\star)}{(x(t),y(t))-(x^\star,y^\star)} \\
+ \psi_r\pa{\anorm{(x(t),y(t))-(x^\star,y^\star)}} .
\end{multline*}
On the other hand, the optimality conditions \eqref{opt_system} tells us that
\[
\dotp{\nabla F(x^\star,y^\star)}{(x(t),y(t))-(x^\star,y^\star)} = -\dotp{\lambda^\star}{Ax(t)+By(t))-c} 
\]
and obviously
\[
F(x^\star,y^\star) = \cL(x^\star,y^\star,\lambda^\star) .
\]
Thus,
\begin{eqnarray*} 
\psi_r\pa{\anorm{(x(t),y(t))-(x^\star,y^\star)}}  \leq \cL(x(t),y(t),\lambda^\star) - \cL(x^\star,y^\star,\lambda^\star) .
\end{eqnarray*}
Invoking the estimate in Theorem~\ref{ACFR,rescale}\ref{theoACFRrescale:itemb}\ref{theoACFRrescale:itembi} yields the claim. \qed
\end{proof}

\begin{remark}\label{rem:convergence_s}
The assumption $\lim_{t\to +\infty} \alpha(t)^2 \sigma(t)^2 b(t) = +\infty$ made in the above theorem is very mild. It holds in particular in all the situations discussed in Section~\ref{sec:particular}. In particular, for $\alpha (t) = t^r$, $0\leq r<1$, $\sigma$ constant, and $b(t)= \frac{1}{t^{2r}} \exp \bpa{\frac{1}{1-r}t^{1-r}}$, one has $\alpha(t)\sqrt{b(t)}= \exp\bpa{\frac{1}{2(1-r)}t^{1-r}}$. Thus, if $F$ is strongly convex, the trajectory $t  \mapsto (x(t),y(t))$ converges exponentially fast to the unique minimizer $(x^\star,y^\star)$.
\end{remark}

\if
{
\begin{remark}

Let us verify on an elementary situation that indeed under the assumptions of theorems 1 and 2, the coupling term is effective to ensure the convergence of each trajectories towards an equilibrium.
Consider $ f = g = $ 0, and $ A = -B = I $, $ c = 0 $ in which case there is a continuum of solutions which is the entire space. The constraint is equivalent to $ x = y $. The system \eqref{eq:trials} is written
\begin{equation*}
\left\{\begin{array}{lll}
\; \ddot x+\gamma (t) \dot x + b(t) \Big(
\lambda  +   \alpha(t) \dot\lambda 
+ \mu (x-y)   \Big) &=&0\vspace{2mm}\\
 \;   \ddot y+\gamma (t)\dot y - b(t)\Big(
\lambda  +   \alpha(t) \dot\lambda 
+ \mu (x-y)  \Big)  &=&0 \vspace{2mm} \\
\;  \ddot \lambda+\gamma (t)\dot \lambda - b(t)
 \Big( x + \alpha(t)\dot x -(y + \alpha(t)\dot y) \Big)&=&0.
 \end{array}\right.
\end{equation*}
\end{remark}
By adding the first two  equations, and defining $s=x+y$, we get
$$
\ddot{s}+\gamma (t) \dot{s}=0.
$$
So, under the condition $(H_0)$, we get that the limit of  $s(t)=x(t)+y(t)$ exists, as $t\to +\infty$.
Moreover by Theorem 1, we know that the limit of $Ax(t)+By(t)$ is equal to zero, under the assumption $\lim_{t\to +\infty} \alpha(t)^2 \sigma(t)^2 b(t)=+\infty$.  According to $A=-B=I$ this implies that the limit
of $x-y$ is zero. Thus $x(t)+y(t)$ and $x(t)-y(t)$ converge, which implies that $x(t)$ and $y(t)$ converge, with the same limit (since their difference tends to zero).
On the other hand, even in this elementary situation, the convergence of 
$\lambda(t)$ is  a non-trivial question.

}
\fi


\section{Parameters choice for fast convergence rates}\label{sec:particular}
In this section, we suppose  that the solution set  $\sS$ of the saddle value problem \eqref{opt_system} is non-empty, so as to invoke Theorem~\ref{ACFR,rescale}\ref{theoACFRrescale:itemb}, Theorem~\ref{ACFR_rescale_boundedness} and Theorem~\ref{Lyap_gen}.
The set of conditions \ref{cond:G1+}, \ref{cond:G2}, \ref{cond:G3}, \ref{cond:G4} and \ref{cond:G5} imposes sufficient assumptions on the coefficients $\gamma, \,  \alpha,  \, b$ of the dynamical system \eqref{eq:trials}, and on the coefficients $\sigma, \delta, \xi$ of the function $\cE$ defined in \eqref{eq:lyapcont}, which guarantee that $\cE$ is a Lyapunov function for the dynamical system \eqref{eq:trials}.

Let us show that this system admits many solutions of practical interest which in turn will entail fast convergence rates. For this, we will organize our discussion around the coefficient $\alpha$ as dictated by Theorem~\ref{Lyap_gen}. Indeed, the latter shows that the convergence rate of the Lagrangian values and feasibility is $\displaystyle{\cO\pa{\exp\pa{-\int_{t_0}^t\frac{1}{\alpha(s)} ds}}}$. Therefore, to obtain a meaningful convergence result, we need to assume that
\[
\int_{t_0}^{+\infty}\frac{1}{\alpha(s)} ds = +\infty.
\]
This means that the critical growth is $\alpha(t) = a t$ for $a > 0$. If $\alpha(t)$ grows faster, our analysis do not provide an instructive convergence rate. So, it is an essential ingredient of our approach to assume that  $\alpha (t)$ remains positive, but not too large as $t\to +\infty$.
In fact, the set of conditions \ref{cond:G1+}, \ref{cond:G2}, \ref{cond:G3}, \ref{cond:G4} and \ref{cond:G5} simplifies considerably by taking $\sigma$ a positive constant, and 
$\gamma\alpha -\dot \alpha$ a constant strictly greater than one. This is made precise in the following statement whose proof is immediate.
\begin{corollary}\label{cor:paramchoice}
Suppose that $\sigma \equiv \sigma_0$ is a positive constant, and $\gamma\alpha -\dot \alpha \equiv \eta>1$.
Then the set of conditions \ref{cond:G1+}, \ref{cond:G2}, \ref{cond:G3}, \ref{cond:G4} and \ref{cond:G5} reduces to
\begin{equation}\label{eq:allconds}
b(1 + 2\eta - 2 \gamma \alpha) - \alpha \dot{b} = 0 \text{ and } \inf_{t\geq t_0}\alpha (t)>0. 
\end{equation}
\end{corollary}
Following the above discussion, we are led to consider the following three cases.
\subsection{Constant parameter $\alpha$}
Consider the simple situation where $\sigma \equiv \sigma_0 > 0$ and $\eta > 1$ in which case \ref{cond:G1+} reads $\gamma\alpha - \dot{\alpha} - 1 = \eta - 1 > 0$. Taking $\alpha \equiv \alpha_0$, a positive constant, yields $\gamma \equiv \frac{\eta}{\alpha_0}$ and \eqref{eq:allconds} amounts to solving 
\[
b - \alpha_0 \dot b = 0 ,
\]
that is, $b(t)= \exp\pa{{\frac{t}{\alpha_0}}}$.
%
Capitalizing on Corollary~\ref{cor:paramchoice}, and specializing the results of Theorem~\ref{ACFR,rescale}\ref{theoACFRrescale:itemb} (or equivalently Theorem~\ref{Lyap_gen} according to Remark~\ref{rem:Lyap_gen}) and Theorem~\ref{ACFR_rescale_boundedness} to the current choice of parameters yields the following statement.

\begin{proposition} \label{O-exp}
Suppose that $\sigma \equiv \sigma_0 > 0$, $\eta > 1$, and that the coefficients of \eqref{eq:trials} satisfy:
the functions $\alpha, \gamma$ are constant with 
\[
\alpha \equiv \alpha_0 >0, \; \gamma \equiv \frac{\eta}{\alpha_0}, \; b(t)= \exp\pa{{\frac{t}{\alpha_0}}}.
\]
Suppose that $\sS$ is non-empty.
Then, for any solution trajectory  $(x(\cdot),y(\cdot),\lambda(\cdot))$ of \eqref{eq:trials}, the trajectory and its velocity remain bounded, and we have 
\begin{align*}
\cL(x(t),y(t),\lambda^\star) - \cL(x^\star,y^\star,\lambda^\star) &= \cO\pa{\exp\pa{-\frac{t}{\alpha_0}}}, \\
\norm{Ax(t)+By(t)-c}^2 &= \cO\pa{\exp\pa{-\frac{t}{\alpha_0}}} , \\
-C_1\exp\pa{-\frac{t}{2\alpha_0}} \leq F( x(t),y(t))-F^\star &\leq C_2\exp\pa{-\frac{t}{\alpha_0}} , \\
\anorm{(\dot {x}(\cdot), \dot {y}(\cdot), \dot{\lambda}(\cdot))} &\in L^2([t_0,+\infty[) ,
\end{align*}
where $C_1$ and $C_2$ are positive constants.
\end{proposition}
%
\subsection{Linearly increasing parameter $\alpha$}\label{sec:var}
We now take $\sigma \equiv \sigma_0 > 0$, $\eta > 1$ and $\alpha(t)= \alpha_{0} t$ with $\alpha_{0} > 0$. Then \ref{cond:G1+} is satisfied and we have $\gamma(t) =  \frac{\eta+\alpha_0}{\alpha_0 t}$. Condition~\eqref{eq:allconds} then becomes
\[
b(t) (1 - 2\alpha_0) - \alpha_0 t \dot b(t) = 0,
\]
which admits $b(t)= t^{\frac{1}{\alpha_0}-2}$ as a solution.
%
We then have $b \equiv 1$ for $\alpha_0=1/2$, while one can distinguish two regimes for its limiting behaviour with
\[
\lim_{t \to +\infty} b(t) = 
\begin{cases}
+\infty & \alpha_{0} < \demi , \\
0 & \alpha_{0} > \demi .
\end{cases}
\]
In view of Theorem~\ref{ACFR,rescale}\ref{theoACFRrescale:itemb} and Theorem~\ref{ACFR_rescale_boundedness}, we obtain the following result.
\begin{proposition} \label{O-1/t2}
Suppose that $\sigma \equiv \sigma_0 > 0$, $\eta > 1$, and that the coefficients of \eqref{eq:trials} satisfy
\[
\alpha(t)= \alpha_{0} t   \mbox{ with } \alpha_{0} >0, \; \gamma(t) =  \frac{\eta+\alpha_0}{\alpha_0 t}, \; b(t)= t^{\frac{1}{\alpha_0}-2}.
\]
Suppose that $\sS$ is non-empty. Then, for any  solution trajectory $(x(\cdot),y(\cdot),\lambda(\cdot))$ of \eqref{eq:trials}, the trajectory remains bounded, and we have the following convergence rates:
\begin{align*}
\cL(x(t),y(t),\lambda^\star) - \cL(x^\star,y^\star,\lambda^\star) &= \cO\pa{\frac{1}{t^{\frac{1}{\alpha_0}}}}, \\
\norm{Ax(t)+By(t)-c}^2 &= \cO\pa{\frac{1}{t^{\frac{1}{\alpha_0}}}} , \\
-\frac{C_1}{t^{\frac{1}{2\alpha_0}}} \leq F( x(t),y(t))-F^\star &\leq \frac{C_2}{t^{\frac{1}{\alpha_0}}} , \\
\anorm{(\dot {x}(t), \dot {y}(t), \dot{\lambda}(t)} &= \cO\pa{\dfrac1{t}} .
\end{align*}
where $C_1$ and $C_2$ are positive constants.
\end{proposition}
%

%

\subsection{Power-type parameter  $\alpha$}
Let us now take $\sigma \equiv \sigma_0 > 0$, $\eta > 1$ and consider the intermediate case between the two previous situations, where $\alpha (t) = t^r$, $0<r<1$. Thus \ref{cond:G1+} is satisfied and we have $\gamma(t) = \frac{\eta}{t^r} + \frac{r}{t}$. Condition~\eqref{eq:allconds} is then equivalent to
\[
b(t) (1- 2r t^{r-1} ) - t^r \dot b(t) = 0,
\]
which, after integration, shows that 
$b(t)= \frac{1}{t^{2r}} \exp\pa{ \frac{1}{1-r}t^{1-r}}$ is a solution. Appealing again to Theorem~\ref{ACFR,rescale}\ref{theoACFRrescale:itemb} and Theorem~\ref{ACFR_rescale_boundedness}, we obtain the following claim.
\begin{proposition} \label{alpha-puissance}
Take $\sigma \equiv \sigma_0 > 0$ and $\eta > 1$. Suppose that the coefficients of \eqref{eq:trials} satisfy
\[
\alpha(t)= t^r   \mbox{ with } 0<r<1, \; \gamma(t) =  \frac{\eta}{t^r} + \frac{r}{t}, \;
b(t)= \frac{1}{t^{2r}} \exp\pa{ \frac{1}{1-r}t^{1-r}} .
\]
Suppose that $\sS$ is non-empty. Then, for any  solution trajectory $(x(\cdot),y(\cdot),\lambda(\cdot))$ of  \eqref{eq:trials}, the trajectory remains bounded, and we have the  convergence rates:
\begin{align*}
\cL(x(t),y(t),\lambda^\star) - \cL(x^\star,y^\star,\lambda^\star) &= \cO\pa{\exp\pa{-\frac{1}{1-r}t^{1-r}}}, \\
\norm{Ax(t)+By(t)-c}^2 &= \cO\pa{\exp\pa{-\frac{1}{1-r}t^{1-r}}} , \\
-C_1\exp\pa{-\frac{1}{2(1-r)}t^{1-r}} \leq F( x(t),y(t))-F^\star &\leq C_2\exp\pa{-\frac{1}{1-r}t^{1-r}} , \\
\anorm{(\dot {x}(t), \dot {y}(t), \dot{\lambda}(t)} &= \cO\pa{\dfrac1{t^r}} .
\end{align*}
where $C_1$ and $C_2$ are positive constants.
\end{proposition}
\if{
\begin{remark}
In the above results, we see again the trade-off between the damping parameter $\gamma$ and the extrapolation parameter $\alpha$.
This phenomenon was  observed by the authors in \cite{ACR-Optimization-2020} in the study of  a third order dynamic for optimization.
\end{remark}
}
\fi


\subsection{Numerical experiments}\label{sec:numerics}
To support our theoretical claims, we consider in this section two numerical examples with $\cX= \cY= \cZ= \R^2$, one with a strongly convex objective $F$ and one where $F$ is convex but not strongly so.
\begin{enumerate}[label={\bf{Example}~\arabic*},itemindent=10ex]
\item We consider the quadratic programming problem
\[
\min_{(x,y) \in \R^4} F (x, y) = \anorm{x - (1,1)^\transp}^2 + \anorm{y}^2  \quad  \st y = x + (-x_2,0)^\transp,
\]
whose objective is strongly convex and verifies all required assumptions.

\item We consider the minimization problem
\[
\min_{(x,y) \in \R^4} F (x, y) = \log\pa{1+\exp\pa{-\dotp{(1,1)^\transp}{x}}} + \anorm{y}^2  \quad  \st y = x + (-x_2,0)^\transp .
\]
The objective is convex (but not strongly so) and smooth as required. This problem is reminiscent of (regularized) logistic regression very popular in machine learning.

\end{enumerate}

In all our numerical experiments, we consider the continuous time dynamical system \eqref{eq:trials}, solved numerically with a Runge-Kutta adaptive method (ode45 in MATLAB) on the time interval $[1, 20]$.

\begin{figure}[htbp]
\begin{subfigure}{\textwidth}
\includegraphics[width=1\textwidth]{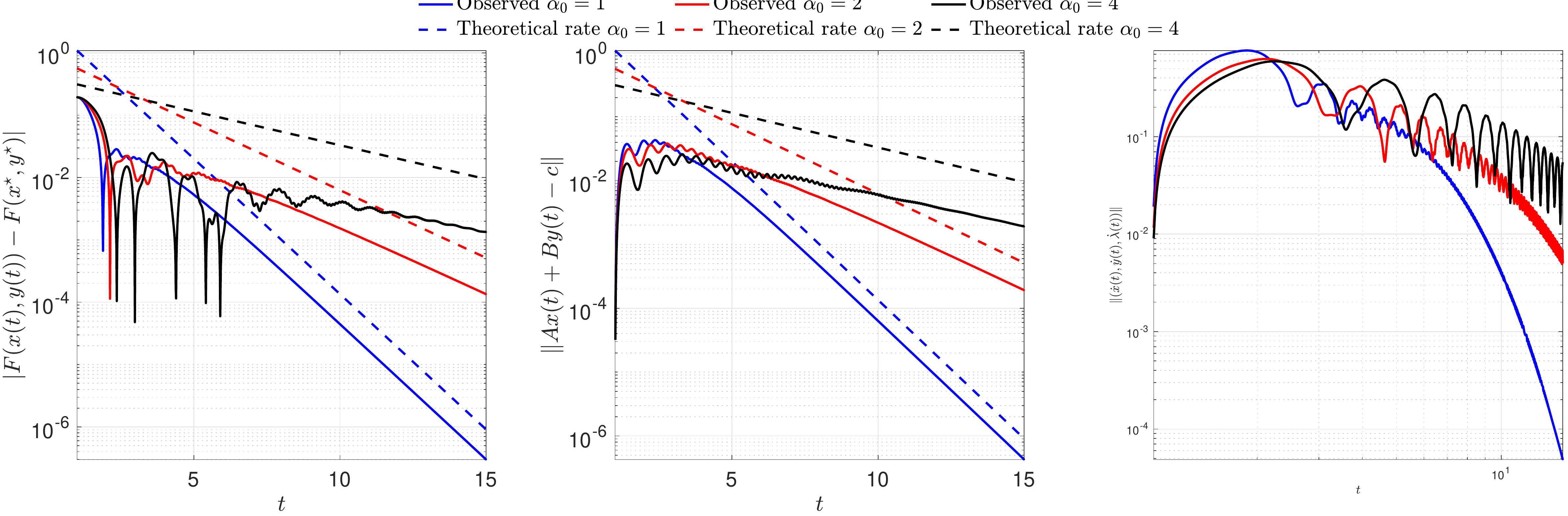}
\caption{$\alpha(t) \equiv \alpha_0$}
\end{subfigure}
\begin{subfigure}{\textwidth}
\includegraphics[width=1\textwidth]{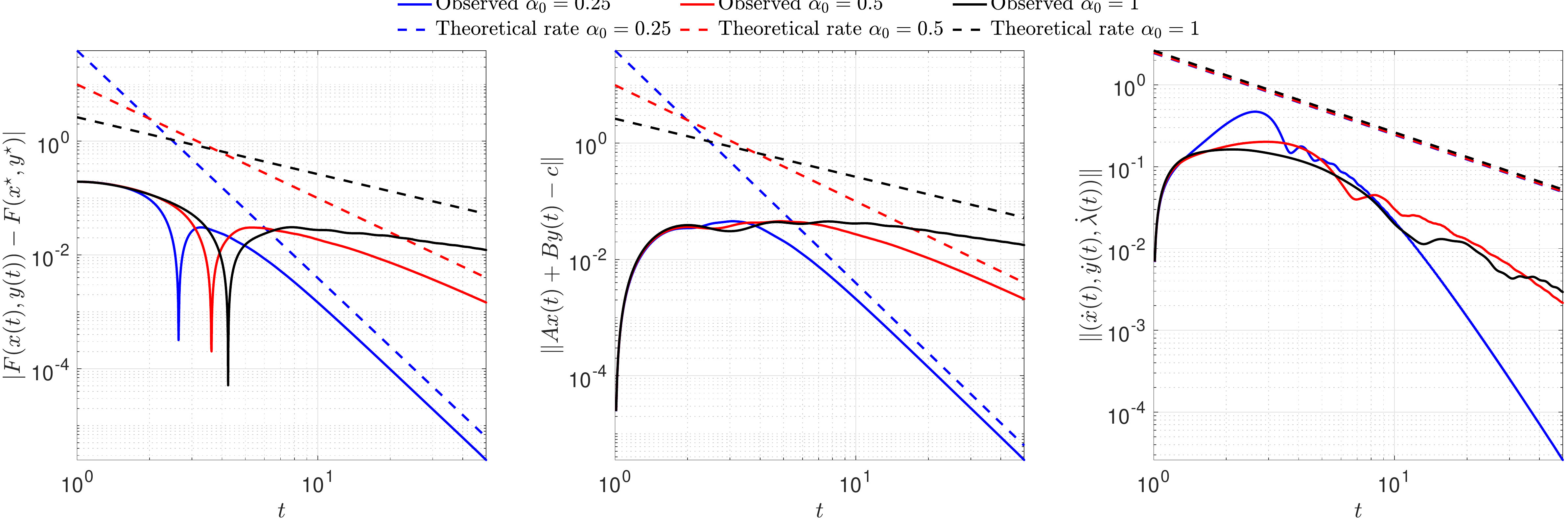}
\caption{$\alpha(t) = \alpha_0 t$}
\end{subfigure}\\
\begin{subfigure}{\textwidth}
\includegraphics[width=1\textwidth]{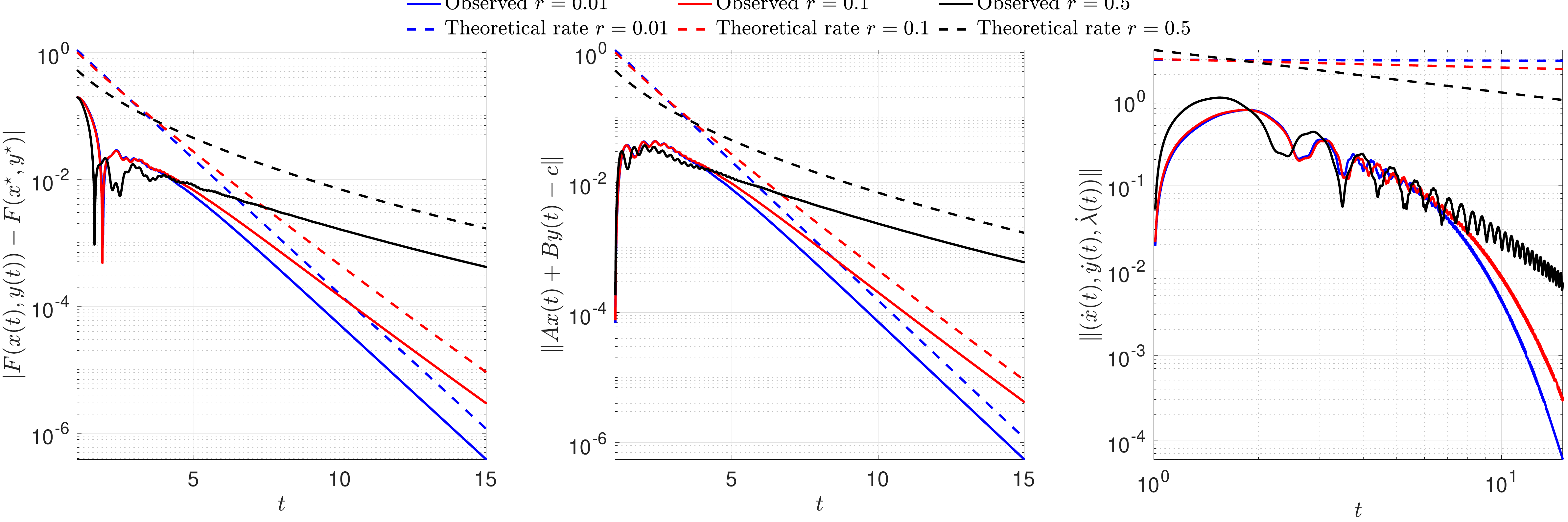}  
\caption{$\alpha(t) = t^r$}
\end{subfigure}

\caption{Experiment on the solely convex objective. Observed (solid) and predicted (dashed) rates on the objective error $\abs{F(x(t),y(t))-F^\star}$ on the left, the feasibility gap $\anorm{Ax(t)+By(t)-c}$ in the middle, and the velocity $\anorm{(\dot {x}(t), \dot {y}(t), \dot{\lambda}(t))}$ on the right.}
\label{fig:trials}    
\end{figure}

\begin{figure}[htbp]
\begin{subfigure}{\textwidth}
\includegraphics[width=1\textwidth]{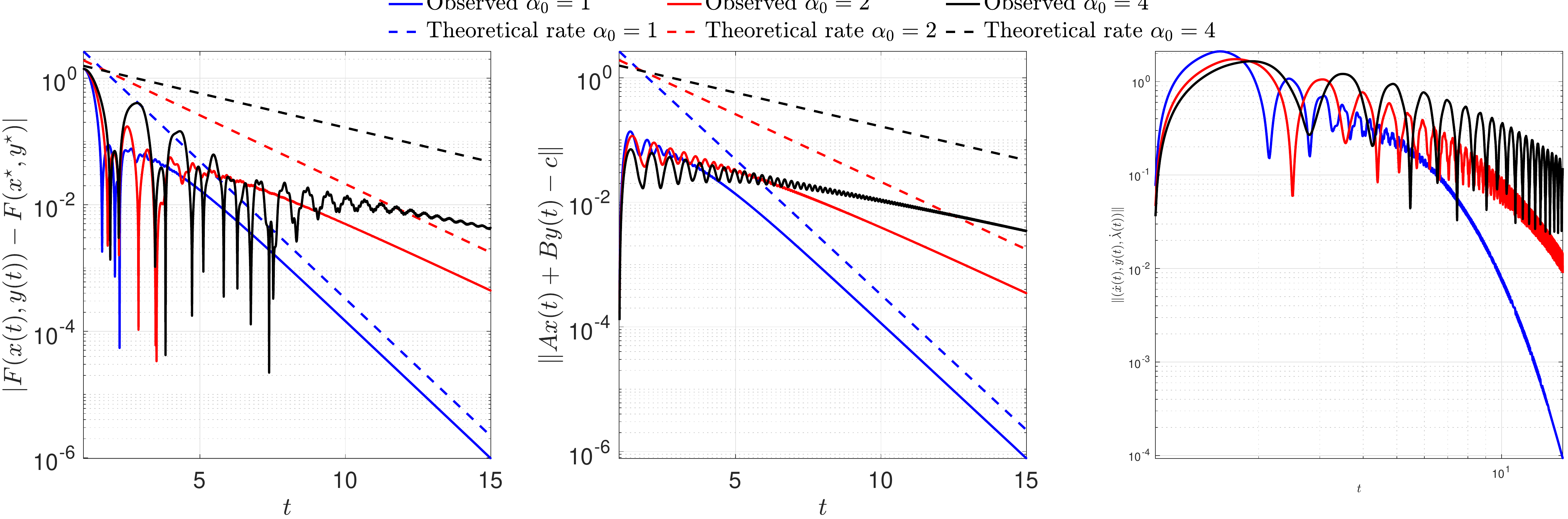}
\caption{$\alpha(t) \equiv \alpha_0$}
\end{subfigure}
\begin{subfigure}{\textwidth}
\includegraphics[width=1\textwidth]{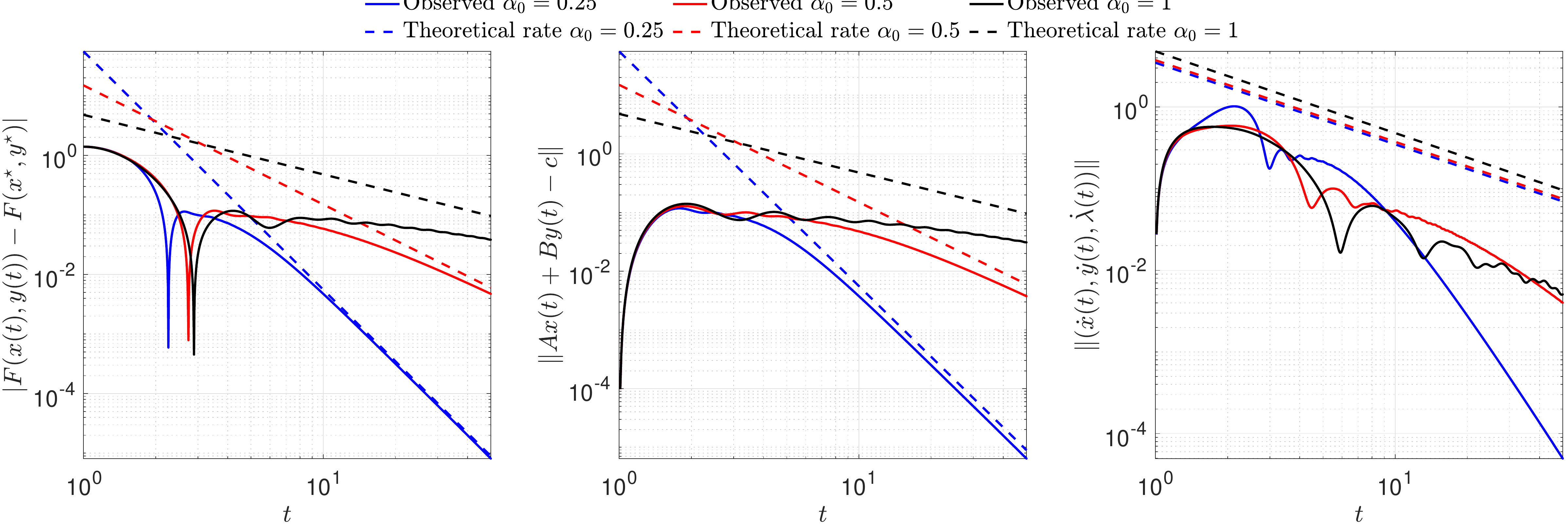}
\caption{$\alpha(t) = \alpha_0 t$}
\end{subfigure}\\
\begin{subfigure}{\textwidth}
\includegraphics[width=1\textwidth]{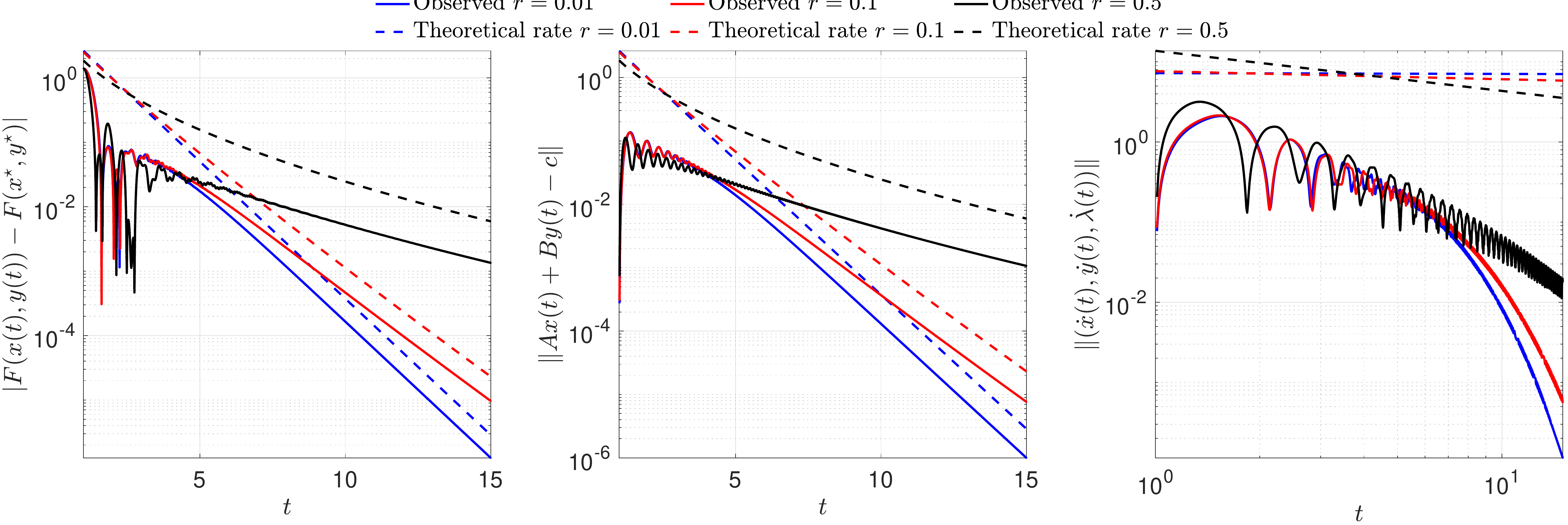}  
\caption{$\alpha(t) = t^r$}
\end{subfigure}

\caption{Experiment on the strongly convex objective. Observed (solid) and predicted (dashed) rates on the objective error $\abs{F(x(t),y(t))-F^\star}$ on the left, the feasibility gap $\anorm{Ax(t)+By(t)-c}$ in the middle, and the velocity $\anorm{(\dot {x}(t), \dot {y}(t), \dot{\lambda}(t))}$ on the right.}
\label{fig:trialsstrong}    
\end{figure}

For the solely convex (resp. strongly convex) objective, Figure~\ref{fig:trials} (resp. Figure~\ref{fig:trialsstrong}) displays the objective error $\abs{F(x(t),y(t))-F^\star}$ on the left, the feasibility gap $\anorm{Ax(t)+By(t)-c}$ in the middle, and the velocity $\anorm{(\dot {x}(t), \dot {y}(t), \dot{\lambda}(t))}$ on the right. In each figure, the first row shows the results for $\alpha(t) \equiv \alpha_0$ with $\alpha_0 \in \bra{1,2,4}$, the second row corresponds to $\alpha(t)=\alpha_0 t$ with $\alpha_0 \in \bra{0.25,0.5,1}$ and the third row to $\alpha(t)=t^r$ with $r \in \bra{0.01,0.1,0.5}$. In all our experiments, we set $\mu=10$ (recall that $\mu$ is the parameter associated with the augmented Lagrangian formulation). All these choices of the parameters comply with the requirements of Propositions~\ref{O-exp}, \ref{O-1/t2} and \ref{alpha-puissance}. The numerical results are in excellent agreement with our theoretical results, where the values, the velocities and the feasibility gap all converge at the predicted rates.

\if
{

\begin{figure}
  \includegraphics[width=1\textwidth]{ADMM-Problem1A_alpha}
  
  \vspace{2mm}
  
 \includegraphics[width=1\textwidth]{ADMM-Problem1B_alpha}
    
  \vspace{2mm}
  
 \includegraphics[width=1\textwidth]{ADMM-Problem1C_alpha}    
    
\end{figure}

\begin{figure}
  
 \includegraphics[width=1\textwidth]{ADMM-Problem2A_alpha}
 
\vspace{2mm}
  
 \includegraphics[width=1\textwidth]{ADMM-Problem2B_alpha}     

  \vspace{2mm}
  
 \includegraphics[width=1\textwidth]{ADMM-Problem2C_alpha}     
       
\caption{\eqref{eq:trials}}
\label{fig:trials}
\end{figure}

}
\fi

\if
{

\noindent \textbf{Example 1}  
Here $ \alpha = 10^{- 0.1} $ and $ \mu = 10^{- 5} $.
In this example, we treat the case with only one variable: 
$$ \min F (x, y) = f (x) + g (y) = \frac12 (x-1)^2 + \frac12y^2  \mbox{ subject to }:  Ax + By = x - y = 1 = c $$
We notice that :

$\bullet$ The convergences of values, constraints and solutions are of the same type, although the values are almost twice as fast as the others.

$\bullet$ From a certain order in time (which is fast) the variations become unchanged.

\medskip

\noindent  \textbf{Example 2}
 Here $ \alpha $ varies and $ \mu = 10^{- 5} $.

This example is similar to Example 1. Note that the variation of $ \alpha $ acts on the three estimates of the values, the constraints and the solutions. More and more $ \alpha $ grows, the variations decrease up to a certain order (here $ \alpha = 10 ^ {- 0.05} $), and then increase until exploding when
 $ \alpha \geq 10^{-2} $.

\medskip

\noindent \textbf{Example 3}
 Here it is similar to Example 2, but with another two-variable function: 
 $$ \min F (x, y) = f (x) + g (y) = \frac12 ( x_1^2 + x_2^2) - \ln (y_1 y_2) $$
  subject to: $ Ax + By = c $ translated by $ x_1 + x_2-y_1 = -1 $ and $ x_1-x_2 + y_1 = 1 $ with $ y_1 , y_2> $ 0.

\medskip

\noindent  \textbf{Example 4}
Similar to Example 3.

\medskip

\noindent  \textbf{Example 5}
 Here we deal with the problem 
 $$ \min F (x, y) = f (x) + g (y) = \frac12 \| x- (1,1)^\top \|^2 ) + \| y
\|^2 $$
 subject to: $ Ax + By = x - y = (1,1)^\top = c $. Note that for two variables the convergences are similar to those in Example 1 with only one variable.

\medskip

\noindent \textbf{ Example 6}
Let us return to the situation considered in Example 5, with a similar problem without constraints $ \min f (x) = \| x- (1,1)^\top \|^2 $. A similar dynamic system provides a faster convergence rate.

 \medskip

Therefore the numerical resolution of the systems associated with problems under constraints is more hampered in time and in estimation than that for those without constraints. This is normal, given that the constraints impose multipliers which make the digital processing more troublesome.

}
\fi

\if
{
\begin{remark}

Let us verify on an elementary situation that indeed under the assumptions of theorems 1 and 2, the coupling term is effective to ensure the convergence of each trajectories towards an equilibrium.
Consider $ f = g = $ 0, and $ A = -B = I $, $ c = 0 $ in which case there is a continuum of solutions which is the entire space. The constraint is equivalent to $ x = y $. The system \eqref{eq:trials} is written
\begin{equation*}
\left\{\begin{array}{lll}
\; \ddot x+\gamma (t) \dot x + b(t) \Big(
\lambda  +   \alpha(t) \dot\lambda 
+ \mu (x-y)   \Big) &=&0\vspace{2mm}\\
 \;   \ddot y+\gamma (t)\dot y - b(t)\Big(
\lambda  +   \alpha(t) \dot\lambda 
+ \mu (x-y)  \Big)  &=&0 \vspace{2mm} \\
\;  \ddot \lambda+\gamma (t)\dot \lambda - b(t)
 \Big( x + \alpha(t)\dot x -(y + \alpha(t)\dot y) \Big)&=&0.
 \end{array}\right.
\end{equation*}
\end{remark}
By adding the first two  equations, and defining $s=x+y$, we get
$$
\ddot{s}+\gamma (t) \dot{s}=0.
$$
So, under the condition $(H_0)$, we get that the limit of  $s(t)=x(t)+y(t)$ exists, as $t\to +\infty$.
Moreover by Theorem 1, we know that the limit of $Ax(t)+By(t)$ is equal to zero, under the assumption $\lim_{t\to +\infty} \alpha(t)^2 \sigma(t)^2 b(t)=+\infty$.  According to $A=-B=I$ this implies that the limit
of $x-y$ is zero. Thus $x(t)+y(t)$ and $x(t)-y(t)$ converge, which implies that $x(t)$ and $y(t)$ converge, with the same limit (since their difference tends to zero).
On the other hand, even in this elementary situation, the convergence of 
$\lambda(t)$ is  a non-trivial question.

}
\fi

\section{Conclusion, perspectives}\label{sec:conclusion}
In this paper, we adopted a dynamical system perspective and we have proposed a second-order inertial system enjoying provably fast convergence rates to solve structured convex optimization problems with an affine constraint. One of the most original aspects of our study is the introduction of a damped inertial dynamic involving several time-dependent parameters with specific properties. They allow to consider a variable viscosity coefficient (possibly vanishing so making the link with the Nesterov accelerated gradient method), as well as variable extrapolation parameters (possibly large) and time scaling.
The analysis of the subtle and intricate interplay between these objects together has been made possible through Lyapunov's analysis.
It would have been quite difficult to undertake such an analysis directly on the algorithmic discrete form. 
On the other hand, as we have now gained a deeper understanding with such a powerful continuous-time framework, we believe this will serve us as a guide to design and analyze a class of inertial ADMM algorithms which can be naturally obtained by appropriate discretization of the dynamics \eqref{eq:trials}. Their full study would go beyond the scope of this paper and will be the subject of future work. Besides, several other open questions remain to be studied, among which, the introduction of geometric damping controlled by the Hessian, and the convergence of the trajectories in the general convex constrained case.


\bibliographystyle{spmpsci}
\bibliography{biblio}

\begin{thebibliography}{10}
\providecommand{\url}[1]{{#1}}
\providecommand{\urlprefix}{URL }
\expandafter\ifx\csname urlstyle\endcsname\relax
  \providecommand{\doi}[1]{DOI~\discretionary{}{}{}#1}\else
  \providecommand{\doi}{DOI~\discretionary{}{}{}\begingroup
  \urlstyle{rm}\Url}\fi

\bibitem{Alv0}
Alvarez, F.: On the minimizing property of a second order dissipative system in
  {H}ilbert spaces.
\newblock SIAM J. Control Optim. \textbf{38}(4), 1102--1119 (2000).
\newblock \doi{10.1137/S0363012998335802}.
\newblock \urlprefix\url{https://doi.org/10.1137/S0363012998335802}

\bibitem{AABR}
Alvarez, F., Attouch, H., Bolte, J., Redont, P.: A second-order gradient-like
  dissipative dynamical system with {H}essian-driven damping. {A}pplication to
  optimization and mechanics.
\newblock J. Math. Pures Appl. (9) \textbf{81}(8), 747--779 (2002).
\newblock \doi{10.1016/S0021-7824(01)01253-3}.
\newblock \urlprefix\url{https://doi.org/10.1016/S0021-7824(01)01253-3}

\bibitem{AAD}
Apidopoulos, V., Aujol, J.F., Dossal, C.: Convergence rate of inertial
  forward-backward algorithm beyond {N}esterov's rule.
\newblock Math. Program. \textbf{180}(1-2, Ser. A), 137--156 (2020).
\newblock \doi{10.1007/s10107-018-1350-9}.
\newblock \urlprefix\url{https://doi.org/10.1007/s10107-018-1350-9}

\bibitem{AttouchBook}
Attouch, H.: Variational convergence for functions and operators.
\newblock Applicable mathematics series. Pitman Advanced Publishing Program
  (1984).
\newblock \urlprefix\url{https://books.google.fr/books?id=oxGoAAAAIAAJ}

\bibitem{Att1}
Attouch, H.: Fast inertial proximal {ADMM} algorithms for convex structured
  optimization with linear constraint.
\newblock Minimax Theory Appl. \textbf{6}(1), 1--24 (2021)

\bibitem{ABCR}
Attouch, H., Balhag, A., Chbani, Z., Riahi, H.: Fast convex optimization via
  inertial dynamics combining viscous and hessian-driven damping with time
  rescaling.
\newblock Evolution Equations \& Control Theory  (2021).
\newblock
  \urlprefix\url{http://aimsciences.org//article/id/92767e45-ae6f-4ff2-9c62-d7ead2d77844}

\bibitem{AC1}
Attouch, H., Cabot, A.: Asymptotic stabilization of inertial gradient dynamics
  with time-dependent viscosity.
\newblock J. Differential Equations \textbf{263}(9), 5412--5458 (2017).
\newblock \doi{10.1016/j.jde.2017.06.024}.
\newblock \urlprefix\url{https://doi.org/10.1016/j.jde.2017.06.024}

\bibitem{AC2}
Attouch, H., Cabot, A.: Convergence rates of inertial forward-backward
  algorithms.
\newblock SIAM J. Optim. \textbf{28}(1), 849--874 (2018).
\newblock \doi{10.1137/17M1114739}.
\newblock \urlprefix\url{https://doi.org/10.1137/17M1114739}

\bibitem{AC2R-EECT}
Attouch, H., Cabot, A., Chbani, Z., Riahi, H.: Rate of convergence of inertial
  gradient dynamics with time-dependent viscous damping coefficient.
\newblock Evol. Equ. Control Theory \textbf{7}(3), 353--371 (2018).
\newblock \doi{10.3934/eect.2018018}.
\newblock \urlprefix\url{https://doi.org/10.3934/eect.2018018}

\bibitem{ACFR}
Attouch, H., Chbani, Z., Fadili, J., Riahi, H.: First-order optimization
  algorithms via inertial systems with hessian driven damping.
\newblock Math. Program.  (2020).
\newblock \urlprefix\url{https://doi.org/10.1007/s10107-020-01591-1}

\bibitem{ACPR}
Attouch, H., Chbani, Z., Peypouquet, J., Redont, P.: Fast convergence of
  inertial dynamics and algorithms with asymptotic vanishing viscosity.
\newblock Math. Program. \textbf{168}(1-2, Ser. B), 123--175 (2018).
\newblock \doi{10.1007/s10107-016-0992-8}.
\newblock \urlprefix\url{https://doi.org/10.1007/s10107-016-0992-8}

\bibitem{ACR-rescale}
Attouch, H., Chbani, Z., Riahi, H.: Fast proximal methods via time scaling of
  damped inertial dynamics.
\newblock SIAM J. Optim. \textbf{29}(3), 2227--2256 (2019).
\newblock \doi{10.1137/18M1230207}.
\newblock \urlprefix\url{https://doi.org/10.1137/18M1230207}

\bibitem{ACR-subcrit}
Attouch, H., Chbani, Z., Riahi, H.: Rate of convergence of the {N}esterov
  accelerated gradient method in the subcritical case {$\alpha\le3$}.
\newblock ESAIM Control Optim. Calc. Var. \textbf{25}, Paper No. 2, 34 (2019).
\newblock \doi{10.1051/cocv/2017083}.
\newblock \urlprefix\url{https://doi.org/10.1051/cocv/2017083}

\bibitem{ACR-Optimization-2020}
Attouch, H., Chbani, Z., Riahi, H.: Fast convex optimization via a third-order
  in time evolution equation.
\newblock Optimization  (2020).
\newblock Preprint available at hal-02432351

\bibitem{ACR-Pafa-2020}
Attouch, H., Chbani, Z., Riahi, H.: Fast convex optimization via time scaling
  of damped inertial gradient dynamics.
\newblock Pure and Applied Functional Analysis  (2020).
\newblock To appear

\bibitem{ACP}
Attouch, H., Czarnecki, M.O., Peypouquet, J.: Coupling forward-backward with
  penalty schemes and parallel splitting for constrained variational
  inequalities.
\newblock SIAM J. Optim. \textbf{21}(4), 1251--1274 (2011).
\newblock \doi{10.1137/110820300}.
\newblock \urlprefix\url{https://doi.org/10.1137/110820300}

\bibitem{AGR}
Attouch, H., Goudou, X., Redont, P.: The heavy ball with friction method. {I}.
  {T}he continuous dynamical system: global exploration of the local minima of
  a real-valued function by asymptotic analysis of a dissipative dynamical
  system.
\newblock Commun. Contemp. Math. \textbf{2}(1), 1--34 (2000).
\newblock \doi{10.1142/S0219199700000025}.
\newblock \urlprefix\url{https://doi.org/10.1142/S0219199700000025}

\bibitem{AP}
Attouch, H., Peypouquet, J.: The rate of convergence of {N}esterov's
  accelerated forward-backward method is actually faster than {$1/k^2$}.
\newblock SIAM J. Optim. \textbf{26}(3), 1824--1834 (2016).
\newblock \doi{10.1137/15M1046095}.
\newblock \urlprefix\url{https://doi.org/10.1137/15M1046095}

\bibitem{AP-max}
Attouch, H., Peypouquet, J.: Convergence of inertial dynamics and proximal
  algorithms governed by maximally monotone operators.
\newblock Math. Program. \textbf{174}(1-2, Ser. B), 391--432 (2019).
\newblock \doi{10.1007/s10107-018-1252-x}.
\newblock \urlprefix\url{https://doi.org/10.1007/s10107-018-1252-x}

\bibitem{APR}
Attouch, H., Peypouquet, J., Redont, P.: A dynamical approach to an inertial
  forward-backward algorithm for convex minimization.
\newblock SIAM J. Optim. \textbf{24}(1), 232--256 (2014).
\newblock \doi{10.1137/130910294}.
\newblock \urlprefix\url{https://doi.org/10.1137/130910294}

\bibitem{APR2}
Attouch, H., Peypouquet, J., Redont, P.: Fast convex optimization via inertial
  dynamics with {H}essian driven damping.
\newblock J. Differential Equations \textbf{261}(10), 5734--5783 (2016).
\newblock \doi{10.1016/j.jde.2016.08.020}.
\newblock \urlprefix\url{https://doi.org/10.1016/j.jde.2016.08.020}

\bibitem{AS}
Attouch, H., Soueycatt, M.: Augmented {L}agrangian and proximal alternating
  direction methods of multipliers in {H}ilbert spaces. {A}pplications to
  games, {PDE}'s and control.
\newblock Pac. J. Optim. \textbf{5}(1), 17--37 (2009)

\bibitem{AD}
Aujol, J.F., Dossal, C.: Stability of over-relaxations for the forward-backward
  algorithm, application to {FISTA}.
\newblock SIAM J. Optim. \textbf{25}(4), 2408--2433 (2015).
\newblock \doi{10.1137/140994964}.
\newblock \urlprefix\url{https://doi.org/10.1137/140994964}

\bibitem{BauschkeCombettes}
Bauschke, H., Combettes, P.L.: Convex Analysis and Monotone Operator Theory in
  Hilbert Spaces, second edn.
\newblock CMS Books in Mathematics. Springer, New York (2017)

\bibitem{BT}
Beck, A., Teboulle, M.: A fast iterative shrinkage-thresholding algorithm for
  linear inverse problems.
\newblock SIAM J. Imaging Sci. \textbf{2}(1), 183--202 (2009).
\newblock \doi{10.1137/080716542}.
\newblock \urlprefix\url{https://doi.org/10.1137/080716542}

\bibitem{Bot-Cest5}
Bo\c{t}, R.I., Csetnek, E.R.: An inertial alternating direction method of
  multipliers.
\newblock Minimax Theory Appl. \textbf{1}(1), 29--49 (2016)

\bibitem{Bot-Cest2}
Bo\c{t}, R.I., Csetnek, E.R.: An inertial {T}seng's type proximal algorithm for
  nonsmooth and nonconvex optimization problems.
\newblock J. Optim. Theory Appl. \textbf{171}(2), 600--616 (2016).
\newblock \doi{10.1007/s10957-015-0730-z}.
\newblock \urlprefix\url{https://doi.org/10.1007/s10957-015-0730-z}

\bibitem{BCH}
Bo\c{t}, R.I., Csetnek, E.R., Hendrich, C.: Inertial {D}ouglas-{R}achford
  splitting for monotone inclusion problems.
\newblock Appl. Math. Comput. \textbf{256}, 472--487 (2015).
\newblock \doi{10.1016/j.amc.2015.01.017}.
\newblock \urlprefix\url{https://doi.org/10.1016/j.amc.2015.01.017}

\bibitem{BCL}
Bo\c{t}, R.I., Csetnek, E.R., L\'{a}szl\'{o}, S.C.: Second-order dynamical
  systems with penalty terms associated to monotone inclusions.
\newblock Anal. Appl. (Singap.) \textbf{16}(5), 601--622 (2018).
\newblock \doi{10.1142/S0219530518500021}.
\newblock \urlprefix\url{https://doi.org/10.1142/S0219530518500021}

\bibitem{Bre1}
Br\'{e}zis, H.: Op\'{e}rateurs maximaux monotones et semi-groupes de
  contractions dans les espaces de {H}ilbert.
\newblock North-Holland Publishing Co., Amsterdam-London; American Elsevier
  Publishing Co., Inc., New York (1973).
\newblock North-Holland Mathematics Studies, No. 5. Notas de Matem\'{a}tica
  (50)

\bibitem{CD}
Chambolle, A., Dossal, C.: On the convergence of the iterates of the ``fast
  iterative shrinkage/thresholding algorithm''.
\newblock J. Optim. Theory Appl. \textbf{166}(3), 968--982 (2015).
\newblock \doi{10.1007/s10957-015-0746-4}.
\newblock \urlprefix\url{https://doi.org/10.1007/s10957-015-0746-4}

\bibitem{DavisYin16}
Davis, D., Yin, W.: Convergence rate analysis of several splitting schemes.
\newblock In: Splitting methods in communication, imaging, science, and
  engineering, Sci. Comput., pp. 115--163. Springer, Cham (2016)

\bibitem{DavisYin17}
Davis, D., Yin, W.: Faster convergence rates of relaxed {P}eaceman-{R}achford
  and {ADMM} under regularity assumptions.
\newblock Math. Oper. Res. \textbf{42}(3), 783--805 (2017).
\newblock \doi{10.1287/moor.2016.0827}.
\newblock \urlprefix\url{https://doi.org/10.1287/moor.2016.0827}

\bibitem{Goldstein}
Goldstein, T., O'Donoghue, B., Setzer, S., Baraniuk, R.: Fast alternating
  direction optimization methods.
\newblock SIAM J. Imaging Sci. \textbf{7}(3), 1588--1623 (2014).
\newblock \doi{10.1137/120896219}.
\newblock \urlprefix\url{https://doi.org/10.1137/120896219}

\bibitem{haraux91}
Haraux, A.: Syst\`emes dynamiques dissipatifs et applications, \emph{Recherches
  en Math\'{e}matiques Appliqu\'{e}es [Research in Applied Mathematics]},
  vol.~17.
\newblock Masson, Paris (1991)

\bibitem{HHF}
He, X., Hu, R., Fang, Y.: Convergence rates of inertial primal-dual dynamical
  methods for separable convex optimization problems.
\newblock arXiv:2007.12428 (2020)

\bibitem{Kang15}
Kang, M., Kang, M., Jung, M.: Inexact accelerated augmented {Lagrangian}
  methods.
\newblock Computational Optimization and Applications \textbf{62}(2), 373--404
  (2015).
\newblock \doi{10.1007/s10589-015-9742-8}.
\newblock \urlprefix\url{https://doi.org/10.1007/s10589-015-9742-8}

\bibitem{Kang13}
Kang, M., Yun, S., Woo, H., Kang, M.: Accelerated {Bregman} method for linearly
  constrained $\ell_1$--$\ell_2$ minimization.
\newblock Journal of Scientific Computing \textbf{56}(3), 515--534 (2013).
\newblock \doi{10.1007/s10915-013-9686-z}.
\newblock \urlprefix\url{https://doi.org/10.1007/s10915-013-9686-z}

\bibitem{May}
May, R.: Asymptotic for a second-order evolution equation with convex potential
  and vanishing damping term.
\newblock Turkish J. Math. \textbf{41}(3), 681--685 (2017).
\newblock \doi{10.3906/mat-1512-28}.
\newblock \urlprefix\url{https://doi.org/10.3906/mat-1512-28}

\bibitem{MVHN}
Michiels, W., Vyhl\'{\i}dal, T., Huijberts, H., Nijmeijer, H.: Stabilizability
  and stability robustness of state derivative feedback controllers.
\newblock SIAM J. Control Optim. \textbf{47}(6), 3100--3117 (2009).
\newblock \doi{10.1137/070697136}.
\newblock \urlprefix\url{https://doi.org/10.1137/070697136}

\bibitem{Nest4}
Nesterov, Y.: Gradient methods for minimizing composite functions.
\newblock Math. Program. \textbf{140}(1, Ser. B), 125--161 (2013).
\newblock \doi{10.1007/s10107-012-0629-5}.
\newblock \urlprefix\url{https://doi.org/10.1007/s10107-012-0629-5}

\bibitem{Nest1}
Nesterov, Y.E.: A method for solving the convex programming problem with
  convergence rate {$O(1/k\sp{2})$}.
\newblock Dokl. Akad. Nauk SSSR \textbf{269}(3), 543--547 (1983)

\bibitem{PSB}
{Patrinos}, P., {Stella}, L., {Bemporad}, A.: {Douglas-Rachford} splitting:
  Complexity estimates and accelerated variants.
\newblock In: 53rd IEEE Conference on Decision and Control, pp. 4234--4239
  (2014).
\newblock \doi{10.1109/CDC.2014.7040049}

\bibitem{PJ}
Pejcic, I., Jones, C.N.: Accelerated {ADMM} based on accelerated
  {Douglas-Rachford} splitting.
\newblock In: European Control Conference (ECC), pp. 1952--1957. Aalborg,
  Denmark (2016)

\bibitem{Polyak1}
Polyak, B.T.: Some methods of speeding up the convergence of iterative methods.
\newblock \v{Z}. Vy\v{c}isl. Mat i Mat. Fiz. \textbf{4}, 791--803 (1964)

\bibitem{Polyak2}
Polyak, B.T.: Introduction to optimization.
\newblock Translations Series in Mathematics and Engineering. Optimization
  Software, Inc., Publications Division, New York (1987).
\newblock Translated from the Russian, With a foreword by Dimitri P. Bertsekas

\bibitem{PoonLiang}
Poon, C., Liang, J.: Trajectory of alternating direction method of multipliers
  and adaptive acceleration.
\newblock In: 33rd Conference on Neural Information Processing Systems
  (NeurIPS). Vancouver, Canada (2019)

\bibitem{Rock2}
Rockafellar, R.T.: Monotone operators associated with saddle-functions and
  minimax problems.
\newblock In: Nonlinear {F}unctional {A}nalysis ({P}roc. {S}ympos. {P}ure
  {M}ath., {V}ol. {XVIII}, {P}art 1, {C}hicago, {I}ll., 1968), pp. 241--250.
  Amer. Math. Soc., Providence, R.I. (1970)

\bibitem{Rock3}
Rockafellar, R.T.: Augmented {L}agrangians and applications of the proximal
  point algorithm in convex programming.
\newblock Math. Oper. Res. \textbf{1}(2), 97--116 (1976).
\newblock \doi{10.1287/moor.1.2.97}.
\newblock \urlprefix\url{https://doi.org/10.1287/moor.1.2.97}

\bibitem{Rock1}
Rockafellar, R.T.: Monotone operators and the proximal point algorithm.
\newblock SIAM J. Control Optim. \textbf{14}(5), 877--898 (1976).
\newblock \doi{10.1137/0314056}.
\newblock \urlprefix\url{https://doi.org/10.1137/0314056}

\bibitem{SDJS}
Shi, B., Du, S.S., Jordan, M.I., Su, W.J.: Understanding the acceleration
  phenomenon via high-resolution differential equations.
\newblock arXiv:2440124 (2018)

\bibitem{SBC}
Su, W., Boyd, S., Cand\`es, E.J.: A differential equation for modeling
  {N}esterov's accelerated gradient method: theory and insights.
\newblock J. Mach. Learn. Res. \textbf{17}, Paper No. 153, 43 (2016)

\bibitem{WRJ}
Wilson, A.C., Recht, B., Jordan, M.I.: A lyapunov analysis of momentum methods
  in optimization.
\newblock arXiv:1611.02635 (2016)

\bibitem{ZLC}
Zeng, X., Lei, J., Chen, J.: Dynamical primal-dual accelerated method with
  applications to network optimization.
\newblock arXiv:1912.03690 (2019)

\end{thebibliography}

\end{document}